\documentclass[letterpaper,11pt,reqno]{amsart}

\makeatletter
\usepackage{amssymb}
\usepackage{latexsym}
\usepackage{amsbsy}
\usepackage{amsfonts}
\usepackage{hyperref}
\usepackage{graphicx}
\usepackage{scalerel}
\usepackage{enumerate}
\usepackage{enumitem}
\usepackage{mathtools}
\usepackage{color}

\usepackage{tikz}

\def\marginpar#1{\ignorespaces}

\DeclareMathOperator\erfc{erfc}

\def\Sup { S^{\uparrow }}
\def\Sdo { S^{\downarrow }}

\newtheorem{theorem}{Theorem}[section]
\newtheorem{lemma}[theorem]{Lemma}

\newtheorem{corollary}[theorem]{Corollary}

\numberwithin{equation}{section}
\makeatother
\begin{document}
\title[Order statistics of Laplace walk]{Hidden symmetries and limit laws in the extreme order statistics of the Laplace random walk}

\author[Jim Pitman]{{Jim} Pitman}
\address{Department of Statistics, University of California, Berkeley.  
} \email{pitman@stat.berkeley.edu}

\author[Wenpin Tang]{{Wenpin} Tang}
\address{Department of Industrial Engineering and Operations Research, Columbia University. 
} \email{wt2319@columbia.edu}

\date{\today} 
\begin{abstract}
This paper is concerned with the limit laws of the extreme order statistics derived from a symmetric Laplace walk.
We provide two different descriptions of the point process of the limiting extreme order statistics: 
a branching representation and a squared Bessel representation. 
These complementary descriptions expose various hidden symmetries in branching processes and Brownian motion which lie behind some striking formulas found by Schehr and Majumdar \cite{SM12}.
In particular, the Bessel process of dimension $4 = 2+2$ appears in the descriptions as a path decomposition of Brownian motion at a local minimum and the Ray-Knight description of Brownian local times near the minimum.
\end{abstract}
\maketitle

\smallskip
\maketitle

\section{Introduction}

\quad Extreme values of a collection of correlated random variables have found various applications
in probability theory, statistics and physics of disordered systems. 
Recently there has been growing interest in understanding not only the extreme of a collection of random variables, but also the second extreme, the third extreme, and so on.
This problem amounts to studying the extreme order statistics of a large number of correlated random variables.
Examples include the extreme order statistics of a random walk \cite{PT20, SM12}, 
the near-extreme structure of a branching random walk or Brownian motion \cite{ABBS, BD11},
and the local eigenvalue statistics of a random matrix near the spectral edge \cite{TV10, TW94}.

\quad The starting point of this paper is the works of Schehr and Majumdar \cite{SM12, SM14} concerning the extreme order statistics of random walk models in the context of statistical physics.
They started from particular models of random walk such as symmetric walks with continuous increment distributions, and made sustained calculations in these models.
Many of their calculations recover known results in the fluctuation theory of random walks. 
But some of their calculations have led to new limit laws and asymptotic formulas, whose place relative to the standard fluctuation theory is much less obvious. 

To describe some of these results in more detail, 
let $S = (S_k, \, 0 \le k \le n)$, with $S_0 = 0$, and $S_k = \sum_{i = 1}^k X_i$ with 
independent and identically distributed (i.i.d.) increments $X_1, \ldots, X_n$.
For $0 \le k \le n$, let $M_{k,n}:= M_k(\{S_0, \ldots, S_n\})$ be the $k^{th}$ order statistic derived from the steps $(S_0, \ldots, S_n)$ of the walk $S$. 
So
\begin{equation}
\label{eq:Mkn}
\{0 = S_0, S_1, \ldots, S_n \} = \{ M_{k,n}, \, 0 \le k \le n \} \mbox{ with } M_{0,n} \le M_{1,n} \le \cdots \le M_{n,n}.
\end{equation}
Schehr and Majumdar \cite{SM12} focused attention on the {\em spacings} or {\em gaps} between random walk order statistics
\begin{equation}
D_{k,n}: = M_{k,n} - M_{k-1,n} \quad \mbox{for } 1 \le k \le n.
\end{equation}
Schehr and Majumdar 
observed that if the distribution of 
$X$ has a symmetric density with $\mathbb{E}X^2 < \infty$, 
then for each fixed $k = 1,2,\ldots$, 
as $n \to \infty$ the expected spacing $\mathbb{E}D_{k,n}$ 
has a limit, for which they gave an integral expression involving the Fourier transform of the density of $X$.
This led to them to conjecture that the distribution of $D_{k,n}$ may approach that of some limit random variable $D_k$ as $n \to \infty$,
a result which was confirmed in greater generality in \cite{PT20}, as discussed in the next paragraph.
In the particular case of the symmetric Laplace walk whose increments have density
\begin{equation}
\label{eq:Laplace}
\frac{\mathbb{P}(X \in dx)}{dx}= \frac{1}{2} e^{-|x|}, \quad x \in (-\infty, \infty),
\end{equation}
Schehr and Majumdar showed that such a limit distribution of $D_k$ exists for each $k$, with a sequence of densities 
$p_k(v):= \mathbb{P}(D_k \in dv)/dv$ that they characterized by the following generating function \cite[(15)]{SM12}:
\begin{equation}
\label{eq:generatingSM}
\sum_{k =1}^{\infty} p_k(v) z^k = 8ze^{-2v} \frac{u_{+}(z) - u_{-}(z)e^{-2v}}{(u_{+}(z) + u_{-}(z)e^{-2v})^3},
\end{equation}
where $u_{\pm}(z):= \sqrt{1-z} \pm 1$.
As they remarked, extracting an explicit formula for $p_k$ from \eqref{eq:generatingSM} is difficult.
However, assuming that $\sqrt{k/2} D_k$ has a limit in distribution as $k \to \infty$, they derived the formula
$\int_0^{\infty} \sqrt{x} p(\sqrt{x})e^{-\lambda x} dx = (1 + \sqrt{\lambda/2} )^{-3}$, $\lambda >0$ 
for a Laplace transform related to the density $p(x)$ of this limit distribution at $x>0$,
which they inverted to obtain the formula \cite[(1)]{SM12}:
\begin{equation}
\label{eq:pdensity}
p(x) = 4 \left[ \sqrt{\frac{2}{\pi}}(1+2x^2) - x(4x^2+3) e^{-2x^2} \erfc(\sqrt{2x}) \right] \quad \mbox{for } x > 0,
\end{equation}
where $\erfc(x): = 2/\sqrt{\pi} \int_x^\infty e^{-t^2}dt$.
This limit density formula invites an interpretation in terms of Brownian motion, however, Schehr and Majumdar did not offer any explicit construction of a random variable with this density.

\quad In our previous work \cite{PT20}, we showed how known results in the fluctuation theory of random walks imply that for every distribution of increments, there is the convergence of finite-dimensional distributions of gaps between order statistics
$(D_{1,n}, D_{2,n}, \ldots) \stackrel{d}{\longrightarrow} (D_1, D_2, \ldots)$
for a limiting joint distribution of consecutive spacings $(D_k, \, k \ge 1)$, which may be constructed from the Feller chains 
$(\Sup_n, \, n \ge 0)$ and $(\Sdo_n, \, n \ge 0)$ generated by the random walk. The definition of the Feller chains is recalled in Section \ref{sc21}. 
To be more precise, let $W_{k,n}: = M_{k, n} - M_{0,n} = \sum_{j = 1}^k D_{j,n}$,
and let $0 = W_0 \le W_1 \le \cdots$ be the order statistics
\begin{equation}
\label{eq:Wkrep}
W_k:= M_k ( \{ - \Sdo_n, \, n \ge 0 \}  \cup \{\Sup_n, \, n \ge 1 \}),
\end{equation}
derived from the two Feller chains $\Sup$ and $\Sdo$.
Then
\begin{itemize}[itemsep = 3 pt]
\item
For each finite $K$,  there is the convergence in total variation of finite-dimensional distributions of order statistics
\begin{equation*}
(W_{k,n},  \,1 \le k \le K )  \stackrel{tv}{\longrightarrow}( W_k, \, 1 \le k \le K ) \quad \mbox{as } n \to \infty.
\end{equation*}
\item
For each fixed $w > 0$, there is the convergence in total variation of laws of counting processes
$N_{W,n}(v):= \sum_{k=1}^n 1( W_{k,n} \le v)$ and $N_W(v):=\sum_{k=1}^\infty 1(W_k \le v)$
\begin{equation*}
\left(N_{W,n}(v), \, 0 \le v \le w \right) \stackrel{tv}{\longrightarrow} \left(N_W(v),  \, 0 \le v \le w \right) \quad \mbox{as } n \rightarrow \infty.
\end{equation*}
\end{itemize}

\quad The purpose of this article is to expose the rich probabilistic structure of the symmetric Laplace walk underlying the striking formulas \eqref{eq:generatingSM} and \eqref{eq:pdensity}.
This structure involves some hidden symmetries of branching processes and Brownian motion, such as 
Le Gall's branching description of random walks stopped at the first descending ladder time \cite{LeGall89},
the Ray-Knight description of Brownian local times \cite{Knight63, Ray63} and McKean's description of the three-dimensional Bessel process \cite{McKean62}, lie behind these formulas.
The main result is stated as follows.

\begin{theorem}
\label{thm:main}
Let $(S_k, \, k \ge 0)$be a random walk with i.i.d. symmetric Laplace increments with density \eqref{eq:Laplace}.
Let $(W_k, \, k \ge 0)$ be defined by \eqref{eq:Wkrep} as the limiting distribution of $(W_{k,n}, \, k \ge 0)$ with $W_{k,n}:=M_{k,n} - M_{0,n}$.
Then
\begin{enumerate}[itemsep = 3 pt]
\item
(Branching representation) 
$(W_k, \, k \ge 0)$ with $W_0 = 0$ is constructed as
\begin{equation}
\label{eq:WkLaprep}
W_k = \sum_{j=1}^k \frac{\varepsilon_j}{S^{\pm \uparrow}_{2j}} \quad \mbox{for } k \ge 1,
\end{equation}
where $(S^{\pm \uparrow}_{n}, \, n \ge 0)$ is the upward Feller chain derived from a simple symmetric random walk, and $(\varepsilon_j, \, j \ge 1)$ is a sequence of i.i.d. standard exponential variables independent of $(S^{\pm \uparrow}_{n}, \, n \ge 0)$.
Consequently, for each $k = 1,2,\ldots$ the distribution of $D_k:= W_k - W_{k-1} = \varepsilon_k/S^{\pm \uparrow}_{2k}$ is determined by
\begin{equation}
\label{eq:tailDj}
\mathbb{P}(D_k > v) = \sum_{i = 1}^k i P^{2k-2}_0(2,2i) e^{-2iv} \quad \mbox{for } v > 0,
\end{equation}
where $P_0$ is the transition matrix of a simple symmetric random walk on the nonnegative integers with absorption at $0$.
That is,
\begin{equation}
\label{eq:P0}
P_0(i,j) = \frac{1}{2} 1(i>0, \, j = i \pm 1) \quad \mbox{for } i, j \ge 0.
\end{equation}
\item
(Squared Bessel representation)
The counting process $(N_W(v), \, v \ge 0)$ with 
\begin{equation}
N_W(v): = \sum_{k=1}^{\infty} 1(W_k \le v) \quad \mbox{for } v \ge 0,
\end{equation}
is a Cox process driven by $\left(\frac{1}{2} Q_{4}(2\gamma_{2},v), \, v \ge 0 \right)$,
where $Q_4$ is a squared Bessel process of dimension $4$, 
and $\gamma_2$ is a gamma random variable with density $xe^{-x}$, $x > 0$, independent of $Q_4$.
Moreover, the tail probability generating function of $(D_k, \, k \ge 1)$ is
\begin{equation}
\label{eq:tailPGFD}
\sum_{k = 1}^{\infty} \mathbb{P}(D_k > v)z^{k-1} = (\sqrt{1-z} \cosh v + \sinh v)^{-2} \quad \mbox{for } v > 0.
\end{equation}
\end{enumerate}
\end{theorem}

\quad The proof of Theorem \ref{thm:main},  given in Section \ref{sc6}, combines ideas from branching processes, excursion theory and path decompositions of Brownian motion. 
The key idea is to study the point process
\begin{equation}
\label{eq:Ndes}
N_{\tiny \mbox{des}}(v): = \sum_{k = 1}^{\tau^{-} - 1} 1(S_k \le v) \quad \mbox{for } v \ge 0,
\end{equation}
where $\tau^-: = \inf\{k: S_k < 0\}$ is the first descending ladder time of $S$.
It will be shown that the point process $N_W$ can be identified either as a limit of $N_{\tiny \mbox{des}}$ conditioned on $\tau^- > m$ as $m \to \infty$,
or as an i.i.d. superposition of $N_{\tiny \mbox{des}}$.
That yields the two different descriptions of the limiting order statistics $(W_k, \, k \ge 0)$.
These two descriptions are also related to the time change between the squared Bessel process of dimension $4$ and the three-dimensional Bessel process, found by Biane and Yor \cite{BY85}.

\quad From the formulas \eqref{eq:WkLaprep}--\eqref{eq:tailDj} we see that the limiting gap distribution of $D_k$ is a probabilistic mixture of exponential distributions with rates $2i$ for $i = 1, \ldots, k$, where the mixing distribution of $i$ is $\mathbb{P}(S^{\pm \uparrow}_{2k-2} = i)$ for 
$S^{\pm \uparrow}$ the Feller chain of a simple symmetric walk.
It is known \cite{Pitman75} that the scaling limit of the upwardly conditioned simple symmetric random walk $(\Sup_n, \, n \ge 0)$ is the three-dimensional Bessel process $(R_t, \, t \ge 0)$ starting from $R_0=0$.
By the branching representation \eqref{eq:WkLaprep}, we obtain
\begin{equation}
\label{eq:113}
\sqrt{\frac{k}{2}} D_k \stackrel{d}{\longrightarrow} \frac{\varepsilon}{2 \chi_3},
\end{equation}
where $\varepsilon$ is exponential with mean $1$, independent ot $\chi_3: = R_1$ with $\chi_3^2$ is the sum of three independent standard Gaussian variables. 
Using the integral formula for the density of a ratio of independent random variables, 
it is easily verified that 
$$\mathbb{P}\left(\frac{\varepsilon}{2 \chi_3} \in dx\right)/dx = p(x)$$ 
as in \eqref{eq:pdensity}. 
Further information provided by this argument is the Mellin transform of the limit density
\begin{equation}
\label{eq:114}
\mathbb{E}\left( \frac{\varepsilon}{2 \chi_3}\right)^s = 2^{-3s/2} \frac{\Gamma(s+1)\Gamma(3/2-s/2)}{\Gamma(3/2)} = \int_0^\infty x^sp(x)dx \quad \mbox{for } -1 < s < 3.
\end{equation}
As a check, it follows from \cite[(6)]{SM12} or \cite[(3.8)]{PT20} that
\begin{equation*}
\mathbb{E}D_{k,n} = u_{k} + u_{n-k+1} \quad \mbox{for } 1 \le k \le n, 
\end{equation*}
where $u_{m}:= \mathbb{P}(S^{\pm}_{2m} = 0) = \binom{2m}{m} 2^{-2m} \sim \frac{1}{\sqrt{\pi m}}$ as $m \to \infty$, 
with $(S^{\pm}_k, \, k \ge 0)$ a simple symmetric random walk. 
This shows that $\mathbb{E}D_k = \lim_{n \to \infty} \mathbb{E}D_{k,n} = u_k \sim \frac{1}{\sqrt{\pi k}}$ as $k \to \infty$,
which agrees with the asymptotics of $\mathbb{E}D_k$ as $k \to \infty$ implied by the formulas 
\eqref{eq:113}--\eqref{eq:114}.
The generating function \eqref{eq:generatingSM} is an easy consequence of the squared Bessel representation \eqref{eq:tailPGFD} by taking derivative in $v$.
Furthermore, the $2m$-step transition probabilities of $P_0$ defined by \eqref{eq:P0} are
\begin{equation}
P^{2m}_0(2i,2j) = \frac{1(i>0)}{2^{2m}}\left[\binom{2m}{m-i+j} - \binom{2m}{m+i+j} \right]  \quad \mbox{for } i, j \ge 0.
\end{equation}
Thus, the formulas \eqref{eq:tailDj} and \eqref{eq:tailPGFD} from two different descriptions of $(W_k, \, k \ge 0)$
give the following non-trivial agreement formula:
\begin{equation}
\scaleto{\sum_{k = 1}^{\infty} \sum_{i = 1}^k \frac{i e^{-2iv}}{2^{2k-2}} \left[ \binom{2k-2}{k-2+i} - \binom{2k-2}{k+i}\right] z^{k-1}= (\sqrt{1-z} \cosh v + \sinh v)^{-2} \quad \mbox{for } v > 0.}{32 pt}
\end{equation}

\quad The remainder of the paper is organized as follows. 
Section \ref{sc2} recalls some basic results about the Feller chains, squared Bessel processes and Cox processes.
Sections \ref{sc3} and \ref{sc4} provide analysis of the Laplace walk through branching processes.
In Sections \ref{sc5} and \ref{sc6}, we establish the squared Bessel representation of the point processes $N_{\tiny \mbox{des}}$ and $N_W$ by embedding the Laplace walk into a Brownian motion.
Theorem \ref{thm:main} is proved in Section \ref{sc6}.
Finally in Section \ref{sc7}, we study the path decomposition at the minimum of Brownian motion in which the Laplace walk is embedded. 
This provides further insight into how the squared Bessel process of dimension $4 = 2 + 2$ arises.
Below we highlight some main results in this paper.

\begin{center}
\begin{tabular}{ |c|c|c| } 
\hline
Branching representation & Squared Bessel representation & Path decomposition \\
\hline
Corollary \ref{coro:bdcox}, Theorem \ref{thm:GLdescription}& Theorem \ref{thm:SBdescription}, Theorem \ref{thm:Poissoncluster} & Theorem \ref{thm:pathdecomp}, Corollary \ref{coro:last} \\
\hline
\end{tabular} 

\medskip
Table 1. A roadmap of the main results.
\end{center}

\section{Background and preliminaries}
\label{sc2}

\quad This section provides background and a few useful results about the Feller chains, squared Bessel processes
and Cox processes.

\subsection{Renewal cluster representation for the Feller chains}
\label{sc21}
Let $(S_k, \, k \ge 0)$ be a random walk with i.i.d. increments $X_1, X_2, \ldots$,
and $(\Sup_k, \, k \ge 0)$ and $(\Sdo_k, \, k \ge 0)$ be the upward and downward Feller chains.
It is known that the Feller chains $\Sup$ and $\Sdo$ can be understood as the Doob-$h$ transform of the walk $S$ with respect to the harmonic functions 
$h^{\uparrow}(x): = \mathbb{E}\left(\sum_{k = 0}^{\tau^{+} - 1}1(S_k > -x) \right)$ with $\tau^{+}: = \inf\{k > 0: S_k > 0\}$ and $x>0$, and $h^{\downarrow}(x): = \mathbb{E}\left(\sum_{k = 0}^{\tau^{-0} - 1}1(S_k  \le -x) \right)$ with $\tau^{-0}: = \inf\{k > 0: S_k \le 0\}$ and $x<0$.
Here we present a pathwise construction of $(\Sup_k, \, k \ge 0)$ and $(\Sdo_k, \, k \ge 0)$ from the walk $(S_k, \, k \ge 0)$.
This construction with a finite time horizon $n$ was introduced by Feller \cite[XII.8, Lemma 3]{Feller2} to provide a combinatorial proof of Sparre Andersen's identity, and was extended to the infinite horizon in \cite{Bertoin93, PT20}.
Formally, the upward Feller chain $S^{\uparrow}$ is the sequence of partial sums of those increments $X_k$ of the walk $S$ with $S_k > 0$, and the downward Feller chain $S^{\uparrow}$ is the sequence of partial sums of those increments $X_k$ of the walk $S$ with $S_k \le 0$.
Let $N_n^{+}: = \#\{k \le n: S_k > 0\}$ and $N_n^{-}: = n - N_n^{+}$.
The above construction gives partial sum processes $(\Sup_k, \, 0 \le k \le N^{+}_n)$ and $(\Sdo_k, \, 0 \le k \le N^{-}_n)$ of random lengths $N^{+}_n$ and $N^{-}_n$ respectively,
and
\begin{equation}
\label{eq:krecovery}
N_k^{+} = N_{k-1}^{+} + 1(S_k > 0), \quad N_k^{-} = N_{k-1}^{-} + 1(S_k \le 0), \quad S_k = \Sup_{N_k^{+}} + \Sdo_{N_k^{-}}.
\end{equation}
By letting $n \to \infty$, we obtain the two infinite horizon Feller chains with the convention that $S^{\uparrow}_k = \infty$ for 
$k > N^{+}_{\infty}:=\lim_{n \to \infty} N^{+}_n$, 
and $-S^{\downarrow}_k = \infty$ for $k > N^{-}_{\infty}:=\lim_{n \to \infty} N^{-}_n$.

\quad For each fixed $n$, let $M^{\uparrow}_{k,n}:=M_k(S^{\uparrow}_k, \, 0 \le k \le N^{+}_n)$ 
and $M^{\downarrow}_{k,n}:=M_k(-S^{\downarrow}_k, \, 0 \le k \le N^{-}_n)$ be the $k^{th}$ order statistics of 
$(\Sup_k, \, 0 \le k \le N^{+}_n)$ and $(-\Sdo_k, \, 0 \le k \le N^{-}_n)$ respectively. 
All that happens to $\left((M^{\uparrow}_{k,n}, M^{\downarrow}_{k,n}), \, k = 0,1,\ldots\right)$ in incrementing from $n$ to $n+1$ is that one more value is sampled from one or other of the two Feller chains, and either this value $S_{n+1}$ is inserted into $(M^{\uparrow}_{k,n}, \, k = 0,1,\ldots)$ if $S_{n+1} > 0$, 
or $-S_{n+1}$ is inserted into $(M^{\downarrow}_{k,n}, \, k = 0,1,\ldots)$ if $S_{n+1} \le 0$.
To be more precise, when $S_{n+1} > 0$ in the update from $n$ to $n+1$, each value $M^{\uparrow}_{k,n}$ with $M^{\uparrow}_{k,n} \le S_{n+1}$ remains unchanged, and if there are $j$ such values then
$M^{\uparrow}_{j+1} = S_{n+1}$ and $M^{\uparrow}_{j+i} = M^{\uparrow}_{j+i-1,n}$ for each $i \ge 2$.
In terms of the counting process
$N_n^{\uparrow}(v):= \sum_{k = 1}^{N^{+}_n} 1(S_k^{\uparrow} \le v)$, $v \ge 0$,
all that happens in incrementing from $n$ to $n+1$ is that an extra point is added at $S_{n+1}=S^{\uparrow}(N_{n+1}^{+})$ if $S_{n+1} > 0$ while there is no change if $S_{n+1} \le 0$.
As is clear from this description, for each $v> 0$,
\begin{equation*}
\mbox{the sequence } N_n^{\uparrow}(v) \mbox{ is increasing to } N^{\uparrow}(v):=\sum_{k=1}^{\infty}1(S^{\uparrow}_k \le v) \mbox{ as } n \to \infty,
\end{equation*}
which is known to be finite with probability one.
A similar description for updating the counting process $N_{n}^{\downarrow}(v):=\sum_{k = 1}^{N^{-}_n} 1(-S_k^{\downarrow} \le v)$, $v \ge 0$ yields for each $v >0$,
\begin{equation*}
\mbox{the sequence } N_n^{\downarrow}(v) \mbox{ is increasing to } N^{\downarrow}(v):=\sum_{k=1}^{\infty}1(-S^{\downarrow}_k \le v) \mbox{ as } n \to \infty.
\end{equation*}
A more precise description of these limiting point processes $N^{\uparrow}$ and $N^{\downarrow}$ on the positive half line is given as follows.

\quad Let $N_0:=(N_0(v), \, v \ge 0)$ be a point process on $[0,\infty)$ and let $\Delta_0$ be a random variable with values in $[0,\infty]$ and $\mathbb{P}(0< \Delta_0<\infty)>0$, defined on the same probability space as $N_0$, according to some joint distribution with $N_0$.
Call a point process $(N(v), \, v \ge 0)$ a renewal cluster process driven by $(N_0, \Delta_0)$ if for $v \ge 0$,
\begin{equation}
N(v) = \sum_{k = 1}^{\infty} N_k(v - T_k) 1(T_k < \infty) = \sum_{k=1}^{\infty}\sum_{i=1}^{N_k(\infty)} 1(T_k + V_{k,i} \le v),
\end{equation}
where $\left((N_k, \Delta_k), \, k = 1,2,\ldots \right)$ is a sequence of i.i.d. copies of $(N_0, \Delta_0)$, with points 
$(V_{k,i}, \, 1 \le i \le N_k(\infty))$ and $T_k:=\Delta_1 + \cdots + \Delta_k$ is the time of the $k^{th}$ renewal in a renewal process with inter-arrival times distributed as $\Delta_0$.

\quad Renewal cluster processes were introduced by Lewis \cite{Lewis64} with the additional assumption that $\Delta_k$ is independent of $N_k$. 
But this is not always the case for the renewal cluster processes generated by the order statistics of Feller chains, as in the following lemma.
The results in this lemma are largely due to Tanaka \cite{Tanaka89}.
The connection to the Feller chains was provided in \cite{Bertoin93}, and some further clarifications and refinements have been drawn from \cite{Biggins03} regarding the renewal cluster process.

\begin{lemma}[Tanaka's decomposition]
\label{lem:clusterTanaka}
Let $(\Sup_k, \, k \ge 0)$ and $(\Sdo_k, \, k \ge 0)$ be the two Feller chains of the random walk $(S_k, \, k \ge 0)$.
\begin{enumerate}[itemsep = 3 pt]
\item
Define the sequence $(T_k^{\uparrow}, \, k = 0,1,\ldots)$ of strictly ascending future minimum times of $\Sup$ by $T^{\uparrow}_0 = 0$ and for $k \ge 1$,
$T_k^{\uparrow}:= \max \left\{j>T_{k-1}^{\uparrow}: S^{\uparrow}_j = \min_{i > T_{k-1}^{\uparrow}} \Sup_i \right\}$.
Then 
the counting process
\begin{equation}
N^{\uparrow}(v) = \sum_{k = 1}^{\infty}  1(T^{\uparrow}_k < \infty) \sum_{i=0}^{T^{\uparrow}_k - T^{\uparrow}_{k-1} - 1}
1\left(\Sup_{T^{\uparrow}_k-i} \le v \right) \quad \mbox{for }v \ge 0,
\end{equation}
is a renewal cluster process driven by $(N^+, \tau^+)$ for $N^{+}(v):= \sum_{k=0}^{\tau^{+}-1}1(-S_k \le v)$ the occupation process of $-S$ prior to $\tau^{+}$.
\item
Define the sequence $(T_k^{\downarrow}, \, k = 0,1,\ldots)$ of weakly ascending future minimum times of $\Sdo$ by $T^{\downarrow}_0 = 0$ and for $k \ge 1$,
$T_k^{\downarrow}:= \min \left\{j>T_{k-1}^{\downarrow}: S^{\downarrow}_j = \max_{i > T_{k-1}^{\downarrow}} \Sdo_i \right\}$.
Then the counting process
\begin{equation}
N^{\downarrow}(v) = \sum_{k = 1}^{\infty}  1(T^{\downarrow}_k < \infty) \sum_{i=0}^{T^{\downarrow}_k - T^{\downarrow}_{k-1} - 1}
1\left(-\Sdo_{T^{\downarrow}_k-i} \le v \right) \quad \mbox{for }v \ge 0,
\end{equation}
is a renewal cluster process driven by $(N^{-0}, \tau^{-0})$ for $N^{-0}(v):= \sum_{k=0}^{\tau^{-0}-1}1(S_k \le v)$ the occupation process of $S$ prior to $\tau^{-0}$.
\end{enumerate}
\end{lemma}

\subsection{Squared Bessel processes}
\label{sc22}
For each fixed $\delta > 0$ and a fixed or a random level $X \ge 0$, 
let $(Q_{\delta}(X,t), \, t \ge 0)$ denote a $\mbox{BESQ}_\delta(X)$ process, 
that is a squared Bessel process of dimension $\delta$ with initial state $Q_{\delta}(0) = X$.
It is known that $(Q_{\delta}(X,t), \, t \ge 0)$ is for each $\delta \ge 0$ and $X \ge 0$ the unique strong solution to the stochastic integral equation:
\begin{equation}
\label{eq:SBsde}
Q_t = X + 2 \int_0^t \sqrt{Q}_s dB_s + \delta t, \quad t \ge 0,
\end{equation}
where $(B_t, \, t \ge 0)$ is a standard Brownian motion independent of $X$.
For each fixed $\delta \ge 0$ the family of laws of $\mbox{BESQ}_\delta(x)$ indexed by $x \ge 0$ is the family of laws of a strong Markov diffusion process on $[0, \infty)$, whose infinitesimal generator acting on suitable smooth functions is 
$2x \frac{d^2}{dx^2} + \delta \frac{d}{dx}$.
The $\mbox{BES}_{\delta}(X)$ process is the square root of the $\mbox{BESQ}_{\delta}(X^2)$ process.

\quad For positive integer $\delta$ and $X \ge 0$, the process $(Q_{\delta}(X,t), \, t \ge 0)$ may be constructed as 
$Q_{\delta}(X,t) = (\sqrt{X} + B^1_t)^2 + \sum_{i = 2}^\delta (B^i_t)^2$
where $(B^i_t, \, t \ge 0)$ are independent standard Brownian motions.
Pythagoras's theorem shows that for positive integer $\delta$ and $\widehat{\delta}$ these processes enjoy the additivity property that
if $Q_{\delta}(X) = (Q_{\delta}(X,t), \, t \ge 0)$ and $\widehat{Q}_{\widehat{\delta}}(\widehat{X}) = (\widehat{Q}_{\widehat{\delta}}(\widehat{X},t), \, t \ge 0)$ are two independent squares of Bessel processes with the indicated dimensions and starting states,
then there is the identity in distribution of processes on the space $\mathcal{C}[0,\infty)$ of continuous real-valued paths:
\begin{equation}
\label{eq:additiveSB}
Q_{\delta}(X) + \widehat{Q}_{\widehat{\delta}}(\widehat{X}) \stackrel{d}{=} Q_{\delta + \widehat{\delta}}(X + \widehat{X}).
\end{equation}
It was shown by Shiga and Watanabe \cite{SW73} that this property extends to all real $\delta, \widehat{\delta} \ge 0$ and $X, \widehat{X} \ge 0$, which provides an alternative definition of the law of $\mbox{BESQ}_{\delta}(X)$ for $\delta \notin \{1,2,\ldots\}$.

\quad The $\mbox{BESQ}_{0}(X)$ process, started at some level $X \ge 0$ plays an important role in understanding the extreme order statistics of the Laplace walk stopped at the first descending ladder time.
According to the result of Shiga and Watanabe, the distribution of $(Q_{0}(X,t), \, t \ge 0)$ with continuous paths is uniquely determined by the identity in law \eqref{eq:additiveSB} for $\widehat{\delta} = 0$, $\widehat{X} = x$, for any particular $X \ge 0$ and $\delta \in \{1,2, \ldots\}$.
The $\mbox{BESQ}_{0}(X)$ process is also called the Feller diffusion. 
As shown by Feller \cite{Feller51}, $\mbox{BESQ}_{0}(X)$ models the total population mass in a continuous state critical branching diffusion process with initial mass $X$.
It is also well known that the $\mbox{BESQ}_{\delta}(X)$ process for $\delta > 0$ models a similar branching diffusion process with an immigration rate controlled by the parameter $\delta$.
See \cite{PW18, RY99} for further background and references of squared Bessel processes. 

\quad From considerations as above and some further stochastic calculus, there is a systematic method, first developed in \cite{PY82}, to compute explicitly the Laplace transform of $\int_0^{\infty} Q_{\delta}(x,t) \mu(dt)$ for all $\delta \ge 0$, $x \ge 0$ and for any $\mu$ on $[0,\infty)$ such that the integral is finite, in terms of suitable solutions to a Sturm-Liouville equation associated with the measure $\mu$.
In particular, for $\mu$ a mixture of uniform distribution on $[0,v]$ and a Dirac mass at $v$, there is the following formula for the joint Laplace transform of $Q_{\delta}(x,v)$ and $\int_0^v Q_{\delta}(x,u) du$ \cite[(2.k)]{PY82}, which has many repercussions in this work:
\begin{multline}
\label{eq:jointLaplace}
\scaleto{\mathbb{E} \exp\left(-\alpha Q_{\delta}(x,v) - \frac{\beta^2}{2} \int_0^v Q_{\delta}(x,u) du \right) 
= \left(\cosh \beta v +  \frac{2 \alpha}{\beta} \sinh \beta v \right)^{-\delta/2} \exp\left( - \frac{x \beta}{2} \frac{1 + \frac{2 \alpha}{\beta} \coth \beta v}{\frac{2 \alpha}{\beta} + \coth \beta v}\right)}{27 pt}
\end{multline}
for $\alpha > 0$ and $\beta \ne 0$.
See also \cite{GY03, Pitman96} and \cite[Chapter XI, \S 1]{RY99}, for various derivations, applications and developments of these formulas.

\quad We also recall the Ray-Knight theorem \cite{Knight63, Ray63} for Brownian local time processes in terms of squared Bessel processes.
See \cite[Chapter VI]{RY99} and \cite{MR06} for background and laws of various Brownian local time processes. 
\begin{lemma}[Ray-Knight theorem]
Let $(B_t, \, t \ge 0)$ be standard Brownian motion, and $(L(x,t), \, x \in \mathbb{R}, t \ge 0)$ be the bi-continuous local time process of $(B_t, \, t \ge 0)$, normalized as occupation density and for each $t \ge 0$,
\begin{equation*}
\int_0^t g(B_s)ds = \int_{-\infty}^{\infty} g(x)L(x,t)dx,
\end{equation*}
for all nonnegative measurable function $g$. 
Let $T_x: = \inf\{t >0: B_t = x\}$ be the first time at which Brownian motion hits the level $x$.
Then for each fixed $x > 0$, the Brownian local time process up to random time $T_{-x}$ is described as follows:
the process $(L(y, T_{-x}), \, y \ge -x)$ is a Markov process with homogeneous transition probabilities on each of the intervals $[-x, 0]$ and $[0,\infty)$, as a $\mbox{BESQ}_2$ on $[-x,0]$ and a $\mbox{BESQ}_0$ on $[0,\infty)$:
\begin{align}
(L(u-x, T_{-x}), \, 0 \le u \le x) & \stackrel{d}{=} (Q_2(0,u), \, 0 \le u \le x), \label{eq:RK2} \\ 
(L(u, T_{-x}), \, u \ge 0) & \stackrel{d}{=} (Q_0(2x \gamma_1, u), \, u \ge 0), \label{eq:RK0}
\end{align}
where $\gamma_1$ is a standard exponential variable. 
Consequently, the final state $L(0,T_{-x})$ at level $x$ of the first local time process indexed by $0 \le u \le x$ is the initial state of the second one indexed by $u \ge 0$, with $L(0, T_{-x}) \stackrel{d}{=} Q_2(0,x) \stackrel{d}{=} 2x \gamma_1$.
\end{lemma}

\subsection{Cox processes}
\label{sc23}

A point process is a right-continuous nonnegative integer-valued counting process. 
Let $X: = (X(t), \, t \ge 0)$ be a nonnegative stochastic process with right-continuous sample paths.
Call a point process $N: = (N(t), \, t \ge 0)$ a {\em Cox process driven by $X$} if $N$ and $X$ are defined on the same probability space, and conditionally given $X$ the process $N$ is a Poisson process with intensity measure $X(t)\, dt$.
Call $N$ a {\em Cox process} if $N$ has the same distribution as a process so constructed from some random intensity process $X$ on a suitable probability space.
The cumulative intensity process generated by $X$ is the continuous increasing process 
\begin{equation}
I(t): = \int_0^t X(s) \, ds \quad t \ge 0,
\end{equation}
which is assumed to be finite almost surely for each $t > 0$.

\quad Let $N_{\theta}$ denote a Cox process with intensity $\theta X$, so the conditional distribution of $N_{\theta}(t)$ given $X$ is Poisson with mean $\theta I(t)$.
Thus, by conditioning on $I(t)$, the probability generating function of $N_{\theta}(t)$ is given by
\begin{equation}
\label{eq:26}
\mathbb{E}z^{N_{\theta}(t)} = \psi(t, (1-z)\theta) \quad \mbox{for } 0 \le z \le 1,
\end{equation}
where for each $t > 0$ the function $\psi(t, \theta)$ is the Laplace transform of $I(t)$ with argument $\theta$:
\begin{equation}
\label{eq:LaplaceTI}
\psi(t, \theta): = \mathbb{E}e^{-\theta I(t)} \quad  \mbox{for } \theta > 0.
\end{equation}
By the uniqueness theorems for probability generating functions and Laplace transforms, 
this formula, and its straightforward extension to linear combinations of increments of $N_{\theta}(t)$ and $I(t)$, imply the well known fact \cite[Theorem 3.3]{Kal17} that the finite dimensional distributions of a Cox process $N$ determine those of its intensity process $X$, and vice versa. 
Let
\begin{equation}
\label{eq:28}
0 < T_{\theta,1} < T_{\theta,2} < \cdots \mbox{ with } \sum_{k \ge 1} 1(T_{\theta, k} \le t) = N_{\theta}(t) \quad \mbox{for } t \ge 0,
\end{equation}
be a listing of the points in $N_{\theta}$ in increasing order, with the convention that $T_{\theta, k} = \infty$ if $N_{\theta}(\infty) < k$.
Assume that $X$ is defined on a probability space supporting also a Poisson process $N(t)= \sum_{k \ge 1} 1(\gamma_k \le t )$ with rate $1$,
which is independent of $X$.
The $n^{th}$ point $\gamma_n$ of $N$ can be represented as $\gamma_n \stackrel{d}{=} \sum_{i = 1}^n \varepsilon_i$ for a sequence of i.i.d. standard exponential variables $\varepsilon_1, \varepsilon_2, \ldots$
Then the Cox process $N_{\theta}$ may be constructed as in \eqref{eq:28} from the points $T_{\theta,k}$ defined by
\begin{equation}
\label{eq:29}
\int_0^{T_{\theta,k}} X(s) \, ds = \frac{\gamma_k}{\theta} \quad \mbox{or} \quad T_{\theta,k} = I^{-1}\left(\frac{\gamma_k}{\theta} \right),
\end{equation}
where $(I^{-1}(t), \, t \ge 0)$ is the right-continuous inverse of the cumulative intensity process $I$.
The basic duality relation between the random time $T_{\theta, k}$ and the counting process $(N_{\theta}(t), \, t \ge 0)$ gives
\begin{equation*}
\mathbb{P}(T_{\theta,k} > t) = \mathbb{P}(N_{\theta}(t) < k) \quad \mbox{for } k = 0,1,\ldots \mbox{ and } t \ge 0.
\end{equation*}
Hence for each fixed $\theta > 0$, the distribution of the random sequence $(T_{\theta,k}, \, k = 1,2, \ldots)$ is determined by the family of finite dimensional distributions of $N_{\theta}$, and vice versa.
In particular, for $k = 1$ the evaluation of \eqref{eq:26} for $z = 0$ gives the basic formula
\begin{equation}
\label{eq:firsttail}
\mathbb{P}(T_{\theta,1} > t) = \mathbb{P}(N_{\theta}(t) = 0) = \psi(t,\theta).
\end{equation}

\quad The following lemma spells out some less well known formulas for Cox processes. 
\begin{lemma}
Let $N_{\theta}$ be a Cox process driven by $\theta X$, with points $T_{\theta,k}$, 
and $\psi(t, \theta)$ be defined by \eqref{eq:LaplaceTI}.
\begin{enumerate}[itemsep = 3 pt]
\item
For each fixed $t \ge 0$ the sequence of tail probabilities $\mathbb{P}(T_{\theta,k} > t)$ for $k = 1,2,\ldots$ is determined by the generating function
\begin{equation}
\label{eq:211}
\sum_{k = 1}^{\infty} \mathbb{P}(T_{\theta,k} > t) z^{k-1} = \frac{\psi(t,(1-z)\theta)}{1-z} \quad \mbox{for } |z| < 1.
\end{equation}
\item
For each real $r > 0$, there is the generating function for $r^{th}$ moments of $T_{\theta,k}$
\begin{equation}
\sum_{k = 1}^{\infty} \mathbb{E}T^r_{\theta,k} = (1-z)^{-1} \int_0^\infty rt^{r-1} \psi(t,(1-z)\theta) dt,
\end{equation}
where the identity holds for all $|z| \le \epsilon$ provided that one of the expressions is finite at $z = \epsilon$ for some $0 < \epsilon < 1$.
\item
For each $k \ge 1$ the joint distribution of $T_{\theta,k}$ and $X(T_{\theta,k})$ on the event $\{T_{\theta_k} < \infty\}$ is determined by the following formula
\begin{equation}
\mathbb{E}[g(X(T_{\theta,k})) 1(T_{\theta,k} \in dt)] = dt \frac{\theta^k}{(k-1)!} \mathbb{E}\left[X(t)g(X(t)) I(t)^{k-1} e^{-\theta I(t)}\right],
\end{equation}
for all nonnegative measurable functions $g$.
\item
The Laplace transform of $X(T_{\theta,k})$ restricted to $\{T_{\theta, k} \in dt\}$ for $0<t<\infty$ is given by
\begin{equation}
\label{eq:214}
\mathbb{E}e^{-\alpha X(T_{\theta,k})} 1(T_{\theta,k} \in dt) = dt \frac{\theta^k}{(k-1)!} \left( \frac{-d}{-d \alpha}\right) \left(\frac{-d}{d \theta} \right)^{k-1} \psi_2(t; \alpha, \theta),
\end{equation}
where $\psi_2(t; \alpha, \theta): = \mathbb{E}e^{-\alpha X(t) - \theta I(t)}$ is the bivariate Laplace transform of $X(t)$ and $I(t)$.
\item
Assume that $\mathbb{P}$ governs $X$ as a Markov process with homogeneous transition probabilities and with some arbitrary initial distribution.
Then for each $t \ge 0$ the sequence of tail probabilities of spacings between consecutive points of $N_\theta$ is determined by the generating function
\begin{equation}
\label{eq:215}
\scaleto{\sum_{k = 1}^\infty \mathbb{P}(T_{\theta, k} < \infty,  \, T_{\theta, k+1}  - T_{\theta, k} > t) z^{k-1} 
 = \frac{\mathbb{E}\psi_3(X(T_{(1-z)\theta,1}), t, \theta) 1(T_{(1-z)\theta,1} < \infty)}{1-z},}{27 pt}
\end{equation}
where $\psi_3(x,t,\theta):= \mathbb{E}\left(e^{-\theta I(t)} \,|\, X_0 = x\right)$, and the distribution of $X(T_{(1-z)\theta,1})$ on the event $\{T_{(1-z)\theta,1} < \infty\}$ is determined by either (3) or (4) above.
\end{enumerate}
\end{lemma}
\begin{proof}
The key to the generating function identity \eqref{eq:211} is to observe that after multiplying both sides by $(1-z)$,
this identity gives two different expressions for $\mathbb{P}(T_{\theta, G(z)} > t)$, where $G(z)$ is a random variable independent of $X$ and $N_{\theta}$, with the geometric distribution $\mathbb{P}(G(z) = k) = z^{k-1}(1-z)$ for $k \ge 1$.
On the left side this probability is computed by conditioning on $G(z)$.
On the right side it is recomputed using the well known fact that $\gamma_{G(z)} \stackrel{d}{=} \frac{\gamma_1}{1-z}$,
which implies via \eqref{eq:29} that $T_{\theta, G(z)} \stackrel{d}{=} T_{(1-z)\theta ,1}$.
The identity \eqref{eq:215} is proved in the same way using the Markov property of $X$ at time $T_{(1-z)\theta, 1}$ to assist the evaluation on the right side.
The remaining identities are easily checked by differentiating or integrating \eqref{eq:211} and \eqref{eq:215}.
\end{proof}

\section{The Laplace walk stopped at the first descending ladder time}
\label{sc3}

\quad In this section we study the order statistics of the symmetric Laplace walk stopped at the first descending ladder time, which relies on Le Gall's branching description \cite{LeGall89} of random walks.

\quad It is assumed throughout that we are working on the event of probability one that the values of $S_k$ are all distinct.
Since the the increment distribution of the random walk $S$ is symmetric and continuous, the laws of strictly/weakly ascending/descending ladder times are identical, so we will not distinguish between strict and weak ladder times.
The next lemma recalls from \cite{LeGall89} some known distributional properties of a symmetric Laplace walk prior to the first descending ladder time.

\begin{lemma}
\label{lem:Laplacedes}
Let $(S_k, \, k \ge 0)$ be a random walk with i.i.d. symmetric Laplace increments $X_1, X_2, \ldots$,
and $\tau^-: = \inf\{k: S_k < 0\}$ be the first descending ladder time of $S$.
Then
\begin{enumerate}[itemsep = 3 pt]
\item
The probability generating function of $\tau^-$ is 
\begin{equation}
\label{eq:PGFtau}
\mathbb{E}z^{\tau^{-}} = 1 - \sqrt{1-z}.
\end{equation}
\item
$\tau^{-}$ and $S_{\tau^{-}}$ are independent, with $S_{\tau^{-}} \stackrel{d}{=} \varepsilon_1$ the standard exponential distribution.
\item
Conditionally given the event $\{\tau^- > 1\}$, let $\alpha$ be the unique random time at which $S$ attains its minimum on 
$1, \ldots,\tau^{-} - 1$.
Then $(S_{\alpha + k} - S_{\alpha}, \, 0 \le k \le \tau^- - \alpha)$ and $(S_{\alpha - k} - S_{\alpha}, \, 0 \le k \le \alpha)$
are two independent copies of $(S_k, \, 0 \le k \le \tau^-)$.
\end{enumerate}
\end{lemma}

\quad We consider the point process with $(1+\tau^{-})$ points at $\{0 = S_0, S_1, \ldots, S_{\tau^{-}}\}$.
It is easy to see that $S_{\tau^{-}} < 0$ is the smallest point, and $S_{0} = 0$ is the second smallest point of $\{S_0, \ldots, S_{\tau^{-}}\}$.
By Lemma \ref{lem:Laplacedes}, the exponential overshoot $-S_{\tau^{-}}$ is independent of $(S_k, \, 0 \le k \le \tau^{-}-1)$, 
hence also independent of the remaining order statistics of $\{S_0,\ldots, S_{\tau^{-}-1}\}$.
To describe the distribution of these remaining order statistics, observe first that
\begin{equation*}
\mathbb{P}(\tau^{-} = 1) = \mathbb{P}(\tau^{-} > 1) = \frac{1}{2}.
\end{equation*}

Given $\tau^{-} = 1$, there is two-point configuration
$(S_0, S_1 \,| \,\tau^{-} = 1) = (S_0, S_1 \,| \,S_1 < 0) \stackrel{d}{=} (-\varepsilon_1, 0)$,
which is of little interest. 
So we will focus on the distribution of $\{S_0, \ldots, S_{\tau^{-}-1}\}$ after conditioning on the event $\{\tau^{-} > 1\}$.

\quad To simplify notation, let $\nu: = \tau^{-}-1$, and on the conditional probability space $\{\nu \ge  1\}$ 
let $M_k: = M_k(\{S_0, \ldots, S_{\nu}\})$ for $0 \le k \le \nu$ so that
\begin{equation*}
M_0 = 0 < M_1 < \cdots < M_{\nu},
\end{equation*}
are order statistics of $\{0 = S_0, S_1, \ldots, S_{\tau^{-}-1}\}$ given $\tau^{-} > 1$.
The following lemma provides a recursive description of the order statistics $(M_1, \ldots, M_{\nu})$.

\begin{lemma}
\label{lem:branching}
Conditioned on the event $\{\nu \ge 1\}$:
\begin{enumerate}[itemsep = 3 pt]
\item
The probability generating function of $\nu$ is given by
\begin{equation}
\label{eq:PGFnu}
\mathbb{E}(z^{\nu} \, |\, \nu \ge 1) = \frac{2}{z}\bigg(1 - \sqrt{1-z} - \frac{z}{2} \bigg).
\end{equation}
\item
The distribution of $M_1$ is exponential with mean $1/2$, i.e. $M_1 \stackrel{d}{=} \frac{1}{2} \varepsilon_1$.
Moreover, independent of $M_1$, the configuration of the remaining points relative to $M_1$
\begin{equation*}
\{M_2 - M_1, \ldots, M_{\nu} -M_1\}
\end{equation*}
has a distribution which is a mixture of three cases:
\begin{enumerate}[itemsep = 3 pt]
\item[(i)]
empty with probability $1/4$;
\item[(ii)]
equal in distribution to $\{M_1, \ldots, M_{\nu}\}$ with probability $1/2$;
\item[(iii)]
equal in distribution to the union of two independent copies of $\{M_1, \ldots, M_{\nu}\}$ with probability $1/4$.
\end{enumerate}
\end{enumerate}
\end{lemma}
\begin{proof}
For part (1), the probability generating function \eqref{eq:PGFnu} is easily derived from \eqref{eq:PGFtau} by noting that 
$\nu \stackrel{d}{=} (\tau^{-} - 1 \, | \, \tau^{-} > 1)$.
Part (2) is read from Le Gall \cite[Remarque, p.261]{LeGall89}.
It is a consequence of Lemma \ref{lem:Laplacedes} and the memoryless property of the exponential distribution.
Lemma \ref{lem:Laplacedes} gives the branching probabilities $1/4$, $1/2$, $1/4$ according to whether the split at the minimum produces $0, 1$ or $2$ non-trivial fragments, with one Bernoulli$(1/2)$ trial for the fragment before the minimum, and an independent Bernoulli$(1/2)$ trial for the fragment after the minimum. 
The point configuration $\{M_2 - M_1, \ldots, M_{\nu} -M_1\}$ is then seen to be the superposition of a binomial$(2, 1/2)$ number of independent copies of the original configuration $\{M_1, \ldots, M_{\nu}\}$.
\end{proof}

\quad Next we consider the counting process $(N_{\tiny \mbox{des}}(v), \, v \ge 0)$ defined by \eqref{eq:Ndes}.
Using the notations in this section, we write
$N_{\tiny \mbox{des}}(v)=  \sum_{k =1}^{\nu} 1(S_k \le v) = \sum_{k =1}^{\nu} 1(M_k \le v)$.
The idea is to connect the process $(N_{\tiny \mbox{des}}(v), \, v \ge 0)$ with 
\begin{equation}
N_{\tiny \mbox{exc}}(v):= \sum_{k = 1}^{\nu} 1(S_{k-1} \le v, \, S_k > v, \, k \le \nu),
\end{equation}
which is the number of excursions of the path $(S_k, \, 0 \le i \le \nu)$ above level $v$.
As indicated by Le Gall \cite[Remarque ($ii$), p.266]{LeGall89},
it follows from Lemma \ref{lem:branching} that 
the upcrossing counting process $(N_{\tiny \mbox{exc}}(v), \, v \ge 0)$ represents numbers of excursions in a critical binary branching process,
starting with $N_{\tiny \mbox{exc}}(0) = 1$, in which each excursion
\begin{itemize}[itemsep = 3 pt]
\item
splits in two at rate $2 \times 1/4 = 1/2$,
\item
dies at rate $2 \times 1/4 = 1/2$. 
\end{itemize}
Moreover, along the branches of this tree of excursions, there is a Poisson process of marks  at rate $2 \times 1/2 = 1$, corresponding to levels $M_k$ which have exactly one child.
These levels $M_k$ are not noticed as jumps of the process $(N_{\tiny \mbox{exc}}(v), \, v \ge 0)$.
This description is summarized in the following theorem.

\begin{theorem}
\label{thm:bdcox}
Conditioned on the event $\{\nu \ge 1\}$, there is the identity in law of counting processes
\begin{equation}
\label{eq:excbd}
(N_{\tiny \mbox{des}}(v), \, v \ge 0 \, | \, \nu \ge 1) \stackrel{d}{=} (N_{\tiny \mbox{count}}(v), \, v \ge 0),
\end{equation}
where $(N_{\tiny \mbox{count}}(v), \, v \ge 0)$ is constructed as follows.
Let $(N_{\pm}(v), \, v \ge 0)$ be a birth and death process with state space the nonnegative integers,
\begin{itemize}[itemsep = 3 pt]
\item
initial state $N_{\pm}(0) = 1$,
\item
transitions $j \to j+1$ at rate $j/2$, and $j \to j-1$ at rate $j/2$.
\end{itemize}
Write $N_{\pm}(v) = 1 + N_{\tiny \mbox{births}}(v) - N_{\tiny \mbox{deaths}}(v)$,
where $N_{\tiny \mbox{births}}(v)$ and $N_{\tiny \mbox{deaths}}(v)$ are the increasing processes counting numbers of births and deaths in $(N_{\pm}(v), \, v \ge 0)$.
Then $(N_{\tiny \mbox{count}}(v), \, v \ge 0)$ is constructed as
\begin{equation}
N_{\tiny \mbox{count}}(v) = N_{\tiny \mbox{births}}(v) + N_{\tiny \mbox{deaths}}(v) + N_{\tiny \mbox{marks}}(v), \quad v \ge 0,
\end{equation}
where $(N_{\tiny \mbox{marks}}(v), \, v \ge 0)$ is a Cox process with intensity $(N_{\pm}(v), \, v \ge 0)$.
\end{theorem}
\begin{proof}
This follows from the preceding discussion using the basic jump-hold descriptions of birth and death, Poisson and Cox processes.
In particular, the birth and death process $(N_{\pm}(v), \, v \ge 0)$ is constructed so that it has the same distribution as 
$(N_{\tiny \mbox{exc}}(v), \, v \ge 0)$.
\end{proof}

\quad It can be easily derived from Theorem \ref{thm:bdcox} a branching description of $(N_{\tiny \mbox{des}}(v), \, v \ge 0 )$ without conditioning on the event $\{ \nu \ge 1\}$. 
The result will be spelled out in Corollary \ref{coro:bdcox} with further discussions.
Here we give another corollary of Theorem \ref{thm:bdcox}, which does not seem to be obvious without a careful accounting of the distribution of the point process $\{0 = M_0, M_1, \ldots, M_{\nu}\}$.

\begin{corollary}
\label{coro:spacingrev}
Conditioned on the event $\{\nu \ge 1\}$, the spacings $\Delta_k: = M_k - M_{k-1}$, $1 \le k \le \nu$ between points in the range of the stopped Laplace walk are reversible.
That is, for each $n = 1,2, \ldots$,
\begin{equation}
(\Delta_1, \ldots, \Delta_n \,| \, \nu = n) \stackrel{d}{=} (\Delta_n, \ldots, \Delta_1 \,| \, \nu = n),
\end{equation}
and hence for all $1 \le m \le n$,
\begin{equation}
(\Delta_1, \ldots, \Delta_m \,| \, \nu \ge  n) \stackrel{d}{=} (\Delta_{\nu}, \ldots, \Delta_{\nu-m+1} \,| \, \nu \ge n),
\end{equation}
\end{corollary}
\begin{proof}
In terms of counting processes, with the notation in Theorem \ref{thm:bdcox}, the assertion of the corollary is that
\begin{equation*}
(N_{\tiny \mbox{count}}(v), \, 0 \le v \le M_{\nu}) \stackrel{d}{=} (N_{\tiny \mbox{count}}(M_{\nu}) -  N_{\tiny \mbox{count}}(M_{\nu} - v), \, 0 \le v \le M_{\nu}).
\end{equation*}
By Theorem \ref{thm:bdcox}, the process $(N_{\tiny \mbox{count}}(v), \, 0 \le v \le M_{\nu})$ is the sum of three counting processes, the first two counting births and deaths in a birth and death process, and the third a Cox process driven by that birth and death process.
Note that the birth and death process starts at $1$, and its penultimate state is $1$.
It is well known that any birth and death process starting and ending in the same state is reversible.
It is also straightforward that after adding a Cox process driven by such a reversible process, the resulting process is still reversible.
Hence the conclusion.
\end{proof}

\quad One consequence of Corollary \ref{coro:spacingrev} is that the last spacing of the order statistics of $\{S_0, \ldots, S_{\nu}\}$
conditioned on $\nu \ge 1$ has the same exponential distribution as the first:
\begin{equation*}
(\Delta_{\nu} \,|\, \nu \ge 1) \stackrel{d}{=} (\Delta_1 \,|\, \nu \ge 1) \stackrel{d}{=} \frac{1}{2} \varepsilon_1.
\end{equation*}
This identity in law also follows from the memoryless property of the exponential distribution. 
However, it is not easy to see the next level without calculation:
\begin{equation*}
(\Delta_{\nu}, \Delta_{\nu-1} \,|\, \nu \ge 2) \stackrel{d}{=} (\Delta_1, \Delta_2 \,|\, \nu \ge 2) \stackrel{d}{=} (\frac{1}{2} \varepsilon_1, \frac{1}{2 Y'} \varepsilon'_1),
\end{equation*}
where $\varepsilon'_1$ is standard exponential independent of $\varepsilon_1$, and $Y' \in \{1, 2\}$ with $\mathbb{P}(Y' = 1) = 1- \mathbb{P}(Y' = 2) = \frac{2}{3}$.
So Theorem \ref{thm:bdcox} exposes the hidden symmetry presented in Corollary \ref{coro:spacingrev}.

\quad To connect with Theorem \ref{thm:main}, we will show in Section \ref{sc6} that the limiting distribution as $n \to \infty$ of $M_{k,n} - M_{k-1,n}$, 
the $k^{th}$ spacing from the bottom between the order statistics of $(S_i, \, 0 \le i \le n)$ 
coincides with the limiting distribution of $\Delta_k$ given $\nu \ge m$ as $m \to \infty$.
By Theorem \ref{thm:bdcox}, the spacings $\Delta_k$ are constructed as an explicit function of a birth and death process which generates a total number of $\nu$ children.
It is well known that conditioning a critical branching process to create a large number of offsprings induces a limit process in the early
generations which may be described as a branching process with immigration, see e.g. \cite{Aldous91, AP98, Evans92}.
The simplest limit of this kind is the limit distribution of the bottom spacing $\Delta_1$ given that $\nu$ is large.
According to Lemma \ref{lem:branching}, 
the conditional distribution of $\Delta_1$ given $\nu \ge m$ is identically equal to that of $\frac{1}{2} \varepsilon_1$ for all $m \ge 1$, 
so the evaluation of this first limit is easy.
The following lemma identifies further limit laws of $\Delta_k$ given $\nu \ge m$ for large $m$ in terms of a Markov chain embedded in the birth and death process.

\begin{lemma}
\label{lem:idDelta}
Let $(N_{\pm}(v), \, v \ge 0)$ be the birth and death process defined in Theorem \ref{thm:bdcox},
and for $k \ge 1$ let $Y_k:=N_{\pm}(M_{k-1})$, i.e. $Y_k$ is the number of excursions of the walk $(S_i, \, 0 \le i \le \nu)$ above level $M_{k-1}$.
Conditioned on the event $\{\nu \ge 1\}$,
\begin{enumerate}[itemsep = 3 pt]
\item
$(Y_k, \, k \ge 1)$ is a homogeneous Markov chain with initial state $Y_1 = 1$ and transition matrix $\frac{1}{2}(I + P_0)$ on the nonnegative integers, where $I$ is the identity matrix and $P_0$ is defined by \eqref{eq:P0}.
\item
Given $(Y_1, \ldots, Y_k)$ with $Y_k >0$, the lowest $k$ spacings $(\Delta_1, \ldots, \Delta_k)$ between values of $(S_i, \, 0 \le i \le \nu)$ are independent exponential variables with rate $2Y_k$. That is,
\begin{equation}
\label{eq:Deltafinite}
(\Delta_j, \, 1 \le j \le k \,|\, \nu \ge k) \stackrel{d}{=} \left(\frac{\varepsilon_j}{2Y_j}, \, 1 \le j \le k \,\bigg| \, Y_k > 0 \right),
\end{equation}
where $(\varepsilon_j, \, j \ge 1)$ is a sequence of i.i.d. standard exponential variables, independent of the Markov chain $(Y_j, \, j \ge 1)$.
\end{enumerate}
Moreover, for each fixed $k$ the limiting joint distribution of the first $k$ spacings given $\nu \ge m$ for large $m$ is given by
\begin{equation}
\label{eq:Deltainfinite}
(\Delta_j, \, 1 \le j \le k \,|\, \nu \ge m) \stackrel{d}{\longrightarrow} \left(\frac{\varepsilon_j}{2Y^{\uparrow}_j}, \, 1 \le j \le k  \right) \quad \mbox{as } m \to \infty,
\end{equation}
where $(Y^{\uparrow}_j, \, 1 \le j \le k)$ with $Y^{\uparrow}_1 = 1$ is the homogeneous Markov chain obtained as the Doob-h transform of $(Y_j, \, 1 \le j \le k)$ with respect to the harmonic function $h(j) = j$, whose $m$-step transition probabilities are 
\begin{equation}
\label{eq:39}
P^{\uparrow}(i,j) = 2^{-m} (I + P_0)^m(i,j) \frac{j}{i}, \quad \mbox{for } i,j \ge 1.
\end{equation}
\end{lemma}
\begin{proof}
For part (1), the branching probabilities are $1/4$, $1/2$, $1/4$ according to whether a point $M_{k}$ splits into 0, 1 or $2$ non-trivial fragments.
Thus,
\begin{equation*}
Y_{k+1} = Y_k \mbox{ with probability } \frac{1}{2} \quad \mbox{and} \quad Y_{k+1} = Y_k \pm 1 \mbox{ with probability } \frac{1}{4},
\end{equation*}
provided that $Y_k > 0$,
which yields the transition matrix $\frac{1}{2}(I + P_0)$.
Part (2) is a consequence of the jump-hold descriptions of the birth and death process and the Cox process of marks.
The second half of this lemma follows from the fact that $(Y^{\uparrow}_j, \, 1 \le i \le k)$ is the limit in distribution of $(Y_j, \, 1 \le j \le k)$ given $Y_m > 0$ as $m \to \infty$.
\end{proof}

\section{Branching descriptions of the Laplace walk stopped at the first descending ladder time}
\label{sc4}

\quad This section provides further discussions on the branching description of the order statistics $\{M_0, \ldots, M_{\nu}\}$ of the Laplace walk before the first descending ladder time $\tau^{-}:= \nu + 1$.
Throughout this section, let $(Z(t), \, t \ge 0)$ be a birth and death process which evolves according to the total number of individuals alive at time $t$ in a critical binary branching process (CBBP) such that
\begin{itemize}[itemsep = 3 pt]
\item
it starts with $Z(0) \ge 0$ individuals at time $t = 0$;
\item
each individual lives an exponential lifetime with mean $1/2$, and according to a fair coin toss independent of the lifetime, the individual dies and leaves either $0$ or $2$ children;
\item
each individual present at any given time continues according to the same branching mechanism, independent of all other individuals. 
\end{itemize}
The counting process $(Z(t), \, t \ge 0)$ is then a birth and death process with transitions from $j$ to $j-1$ at rate $j$, from $j$ to $j+1$ at rate $j$, and all other transitions at rate $0$.

\quad The following corollary of Theorem \ref{thm:bdcox} gives a representation of $\{M_0, \ldots, M_{\nu}\}$ 
without conditioning on $\{\nu \ge 1\}$.

\begin{corollary}
\label{coro:bdcox}
The order statistics $\{M_0, \ldots, M_{\nu}\}$ of $(S_i, \, 0 \le i \le \nu)$ have the same joint distribution as if they were constructed as 
$M_0 = 0$ and $M_k$ for $k = 1,2,\ldots$ the time of $k^{th}$ birth or death or mark generated by the critical binary birth and death process $(Z(t/2), \, t \ge 0)$ according to the scheme:
\begin{equation}
\label{eq:LeGallscheme}
M_k:= \mbox{the } k^{th}\,\, t > 0: Z\left(\frac{t}{2}\right) - Z\left(\frac{t-}{2}\right) = \pm 1 \mbox{ or } N_{\tiny \mbox{marks}}(t) - N_{\tiny \mbox{marks}}(t-) = 1,
\end{equation}
with $Z(0)$ assigned Bernoulli$(1/2)$ distribution on $\{0,1\}$, and 
\begin{equation}
N_{\tiny \mbox{marks}}(t):= N\left( \int_0^t Z(v/2) dv\right), \quad t \ge 0,
\end{equation}
the Cox process with intensity $Z(t/2)$ at time $t$ derived from $Z$ and an independent standard Poisson process $N$.
Consequently, the probability generating function of $\nu$ is
\begin{equation}
\label{eq:nuPGF}
\mathbb{E}z^{\nu} = z^{-1}(1 - \sqrt{1-z}).
\end{equation}
\end{corollary}

\quad Next we will give an alternative branching description of the point process $\{M_0, \ldots, M_{\nu}\}$. 
We recall a lemma from Feller \cite[XVII.10, 11]{Feller1} which specifies the distribution of $Z(t)$ for any fixed $Z(0)$. 
See also \cite{Tavare18} for related discussions on the linear birth and death process. 
\begin{lemma}
For $k \ge 0$, let $\mathbb{P}_k$ govern the critical binary branching process $(Z(t), \, t \ge 0)$ with $Z(0) = k$.
Then
\begin{equation}
\label{eq:43}
\mathbb{P}_1(Z(t) = 0) = \frac{t}{t+1}\quad \mbox{and} \quad \mathbb{P}_1(Z(t) = n) = \frac{t^{n-1}}{(1+t)^{n+1}} \,\,\, \mbox{for } n \ge 1,
\end{equation}
and for each $k \ge 0$,
\begin{equation}
\label{eq:44}
\mathbb{P}_k(Z(t) = 0) = \left(\frac{t}{1+t}\right)^k.
\end{equation}
\end{lemma}

\quad The following theorem is a consequence of Le Gall \cite{LeGall89}, 
which is also known as Geiger's lifeline representation of a CBBP \cite{Geiger95}, or the binary-$(0,1)$ tree \cite{PW05}.
\begin{theorem}
\label{thm:GLdescription}
The order statistics $\{M_0, \ldots, M_{\nu}\}$ of $(S_i, \, 0 \le i \le \nu)$ have the same joint distribution as if they were constructed as 
$M_0 = 0$ and $M_k$ for $k = 1,2,\ldots$ the time of $k^{th}$ death in the critical binary birth and death process $(Z(t), \, t \ge 0)$:
\begin{equation}
\label{eq:LeGallscheme2}
M_k:= \mbox{the } k^{th}\,\, t > 0: Z(t) - Z(t-) = -1,
\end{equation}
with $Z(0)$ assigned the geometric$(1/2)$ distribution on $\{0,1, \ldots\}$, i.e.
$\mathbb{P}(Z(0) = n) = 2^{-n-1}$ for $n \ge 0$.
Consequently, 
\begin{equation}
\label{eq:maxMnu}
\mathbb{P}(M_{\nu} > t) = \frac{1}{2+t} \quad \mbox{for } t > 0.
\end{equation}
\end{theorem}
\begin{proof}
The construction of $M_k$'s is just a reformulation of \cite[Theorem 3]{LeGall89}.
Moreover, for each $t > 0$,
\begin{align*}
\mathbb{P}(M_{\nu} > t) & = \mathbb{P}(Z(t) > 0) \\
& = 1 - \sum_{n \ge 0} \mathbb{P}(Z(0) = n) \, \mathbb{P}_n(Z(t) = 0) \\
& =1 - \sum_{n \ge 0} 2^{-n-1} \left(\frac{t}{1+t}\right)^n,
\end{align*}
where the last equality follows from \eqref{eq:44}. 
This leads to \eqref{eq:maxMnu}.
\end{proof}

\quad The construction of Theorem \ref{thm:GLdescription} differs from that of Corollary \ref{coro:bdcox} in two ways which compensate each other. 
In Theorem \ref{thm:GLdescription}, the CBBP is started with a geometric$(1/2)$ number of initial individuals with $\mathbb{E}Z(0) = 1$, but only deaths of individuals count towards the counting process which generates a copy of $N_{\tiny \mbox{des}}$.
As deaths occur at rate $j$ when $Z(v) = j$, and $\mathbb{E}Z(v) \equiv 1$ for all $v$, the mean rate of points of $N_{\tiny \mbox{des}}$ per unit level at level $v$ is therefore always $1$.
In Corollary \ref{coro:bdcox}, the CBBP is started with a Bernoulli$(1/2)$ number of initial individuals with $\mathbb{E}Z(0) = 0$.
Moreover, when the driving process $Z(v/2) =j$, deaths of individual only occur at rate $j/2$.
However, this is compensated by the fact that birth times are counted at rate $j/2$, as do additional marks at rate $j$.
So at any particular level $v$, the mean rate of points of $N_{\tiny \mbox{des}}$ per unit level is $\mathbb{E}Z(v)(\frac{1}{2} + \frac{1}{2} + 1) = 1$ as expected.
Furthermore, for the CBBP started with $0$ or $1$ individual, 
\begin{align*}
\mathbb{P}(M_{\nu} > t) = \mathbb{P}(Z(t/2) > 0)  = \frac{1}{2} 0 + \frac{1}{2} \mathbb{P}_1(Z(t/2) > 0) = \frac{1}{2+t},
\end{align*}
where the last equality follows from \eqref{eq:43}. 
This is in agreement with the formula \eqref{eq:maxMnu}.

\quad In contrast with Corollary \ref{coro:bdcox}, it is much less apparent from the representation of $M_k$'s in Theorem \ref{thm:GLdescription} that the spacings $(\Delta_1,\ldots, \Delta_{\nu})$ between these points  are reversible.
It is also less obvious that $M_1:=\min_{1 \le i \le \nu} S_i$ represented as the first death time in the CBBP started with a geometric$(1/2)$ number of individuals $Z(0)$ given $Z(0) \ge 1$ has the exponential distribution with mean $1/2$.
Recall that $\mathbb{P}_1$ governs the CBBP $(Z(t), \, t \ge 0)$ with $Z(0) = 1$.
By the agreement of the law of $M_1$ in both constructions, we get
\begin{align*}
\mathbb{P}(M_1 > t) & = e^{-2t} && (\mbox{by Corollary \ref{coro:bdcox}}) \\
& = \sum_{n \ge 0} 2^{-n-1} \mathbb{P}_1(\mbox{first death time} > t)^n&& (\mbox{by Theorem \ref{thm:GLdescription}})
\end{align*}
which yields the distribution of the first death time in the CBBP $(Z(t), \, t \ge 0)$ started with one individual:
\begin{equation}
\label{eq:firstdeath}
\mathbb{P}_1(\mbox{first death time} > t) = \frac{2 e^{-2t}}{1+e^{-2t}}.
\end{equation}
Clifford and Wei \cite{CW93}, extending earlier work \cite{Sr88}, showed that the death times in a branching process may be described as a Cox process driven by some squared radial Ornstein-Uhlenbeck process. 
But they only considered subcritical branching processes with immigration, 
so the formula \eqref{eq:firstdeath} for the death times in a CBBP does not seem to be easily read from these results.
On the other hand, it is obvious from the representation in Theorem \ref{thm:GLdescription} that given $\nu \ge 1$ the last spacing $\Delta_{\nu}:=M_{\nu} - M_{\nu-1}$ has the exponential distribution with mean $1/2$, because this random variable is the last holding time of $(Z(t), \, t \ge 0)$ in state $1$, which has an exit rate $2$. 

\quad For the CBBP model with $Z(0)$ distributed as geometric$(1/2)$, it is also interesting to consider the distribution of $Z(0)$ given the event
\begin{equation*}
\{Z(t) > 0\} = \{M_{\nu} > 0\},
\end{equation*}
which has probability $\frac{1}{2+t}$.
A simple Bayes calculation using \eqref{eq:44} and \eqref{eq:maxMnu} gives for $n \ge 1$,
\begin{align}
\mathbb{P}(Z(0) = n \,|\, Z(t) > 0) &= 2^{-n-1}(t+2)\left( 1 - \bigg(\frac{t}{1+t}\bigg)^n\right) \notag\\
& \longrightarrow n 2^{-n-1} \quad \mbox{as } t \to \infty.
\end{align}
This is the negative binomial distribution with parameters $2$ and $1/2$ shifted up by $1$.
As a result,
\begin{equation}
(Z(0) \,|\, Z(t) > 0) \stackrel{d}{\longrightarrow} 1 + Z(0) + Z'(0) \quad \mbox{as }t \to \infty, 
\end{equation}
where $Z'(0)$ is an independent copy of $Z(0)$, distributed as geometric$(1/2)$ on $\{0,1,\ldots\}$.
So with conditioning on $\{Z(t) > 0\}$ for large $t$, we get
\begin{equation}
\mathbb{E}(Z(0) \,|\, Z(t ) > 0) \longrightarrow 3 \quad \mbox{as } t \to \infty.
\end{equation}
As mentioned in Section \ref{sc3}, conditioning a critical branching process on nonextinction in the limit induces a single line of descent for an immortal particle which spins off particles continuing with the regular branching mechanism at some constant birth rate. 
The effect of this conditioning is to introduce a immigration term into the branching process.
In the present context with a geometric$(1/2)$ initial number of individuals, one of these individuals is destined to be the immortal particle, while the trees generated by the other two will terminate in finite time almost surely.

\section{Squared Bessel representation of the Laplace walk stopped at the first descending ladder time}
\label{sc5}

\quad Corollary \ref{coro:bdcox} and Theorem \ref{thm:GLdescription} provide descriptions of the point process $\{M_0, \ldots, M_{\nu}\}$ 
with $\nu:= \tau^{-} -1$ which are adequate in many aspects.  
For instance, for any fixed $v \ge 0$, they can be used to simulate all points $M_k$ in this process with $M_k \le v$ in finite expected time, 
which cannot be done by naive simulation of all $\nu$ points of the process as $\mathbb{E} \nu = \infty$.
Still, these descriptions do not easily yield a formula for the distribution of $M_k$, or of the count $N_{\tiny \mbox{des}}(v): = \sum_{k \ge 1} 1(M_k \le v)$ for any fixed $v$, let alone the multivariate distributions of these statistics. 
In this section we give some alternate description of $\{M_0, \ldots, M_{\nu}\}$, which is more convenient for such computations.

\quad The idea is to embed the symmetric Laplace walk $S$ into a standard Brownian motion $(B_t, \, t \ge 0)$:
\begin{equation}
\label{eq:embedding}
S_n = B_{2 \gamma_n} \quad \mbox{for } n \ge 1,
\end{equation}
where $\gamma_n: = \sum_{i = 1}^n \varepsilon_i$ for a sequence of i.i.d. standard exponential variables $\varepsilon_1, \varepsilon_2, \ldots$ independent of the Brownian motion $(B_t, \, t \ge 0)$.
Equivalently, $\gamma_n$ is the $n^{th}$ point in a Poisson process $N(t): = \sum_{n \ge 1} 1(\gamma_n \le t)$ with rate $1$, 
and $\gamma_n$ is gamma$(n,1)$ distributed with density $f_{\gamma_n}(x) = \frac{1}{(n-1)!} x^{n-1} e^{-x}$, $x > 0$.
We will use Brownian excursion theory to study the counting process $(N_{\tiny \mbox{des}}(v), \, v \ge 0)$.

\quad Recall from Section \ref{sc22} that $(Q_{\delta}(X, t), \, t \ge 0 )$ is a squared Bessel process of dimension $\delta$ with initial state $X$.
The following theorem identifies the counting process $(N_{\tiny \mbox{des}}(v), \, v \ge 0)$ as a Cox process driven by some functional of the $\mbox{BESQ}_0$ process. 
It is a variant of the Ray-Knight identity \eqref{eq:RK0} by embedding the symmetric Laplace walk into Brownian motion.
\begin{theorem}
\label{thm:SBdescription}
Let $(S_n, \, n \ge 0)$ be the symmetric Laplace walk represented as in \eqref{eq:embedding}, 
where $\gamma_n$ is the $n^{th}$ point in a Poisson process $N$ with rate $1$ independent of Brownian motion.
Let  $\tau^{-}:=\inf \{k \ge 1: S_k \le 0\}$.
Then:
\begin{enumerate}[itemsep = 3 pt]
\item
There is the identity in law of processes on $\mathcal{C}[0,\infty)$
\begin{equation}
\label{eq:54}
\left(L(v, 2 \gamma_{\tau^-}), \, v \ge 0\right) \stackrel{d}{=} (Q_0(2 \gamma_1, v), \, v \ge 0).
\end{equation}
\item
The counting process $(N_{\tiny \mbox{des}}(v), \, v \ge 0)$ defined by \eqref{eq:Ndes} is a Cox process driven by 
$(\frac{1}{2}L(v, 2 \gamma_{\tau^-}), \, v \ge 0)$ so that
\begin{equation}
\label{eq:55}
(N_{\tiny \mbox{des}}(v), \, v \ge 0) \stackrel{d}{=} \left(N\left(\frac{1}{2} \int_0^v Q_0(2\gamma_1, u)du \right), \, v \ge 0 \right),
\end{equation}
where on the right side the $\mbox{BESQ}_0$ process $Q_0$ with initial state $2 \gamma_1$ is independent of the Poisson process $N$ with rate $1$.
\item
For each fixed $v > 0$, the probability generating function of $N_{\tiny \mbox{des}}(v)$ is
\begin{equation}
\label{eq:Laplace56}
\mathbb{E}z^{N_{\tiny \mbox{des}}(v)} = (1 + \sqrt{1-z} \tanh v\sqrt{1-z})^{-1}.
\end{equation}
\end{enumerate}
\end{theorem}

\begin{proof}
(1) Let $(\tau_{\ell}, \, \ell \ge 0)$ be the right-continuous inverse of $(L(0,t), \, t \ge 0)$.
According to It\^o's excursion theory (see e.g. \cite[Chapter XII]{RY99}), when the excursion of $|B|$ away from $0$ on the random interval 
$(\tau_{\ell-}, \tau_{\ell})$ is indexed by the constant value $\ell$ of $L(0,t)$ for all $t \in (\tau_{\ell-}, \tau_{\ell})$, 
the point processes of positive and negative excursions are independent and identically distributed copies of the Poisson point process of positive excursions. 
This assertion remains true if the excursion of $B$ on $(\tau_{\ell-}, \tau_{\ell})$ is embellished to include the increments of the independent Poisson process $(N(t/2), \, t \in (\tau_{\ell-}, \tau_{\ell}))$ as an auxilliary marking process.
For any random time $T$ such that $B_T < 0$ with probability one, there is the excursion decomposition of local time on the positive half line
\begin{equation}
\label{eq:localtdecomp}
L(v,T) = \sum_{\ell < L(0,T)} (L(v, \tau_{\ell}) - L(v, \tau_{\ell-}))1(B_t > 0 \mbox{ for } t \in  (\tau_{\ell-}, \tau_{\ell})),
\end{equation}
where the sum is over the random countable set of $\ell$ with $\tau_{\ell-} < \tau_{\ell}$, and the only terms which contribute to the sum are
those corresponding to positive excursions. 

\quad Consider now a random time $T$ contained in the first negative excursion of some particular kind, whose rate per unit local time is say $\xi$, the way that
\begin{itemize}[itemsep = 3 pt]
\item
the time $T = T_{-x}$ is contained in the first excursion to reach level $-x < 0$ with rate $\xi = \frac{1}{2x}$ (see e.g. \cite[Chapter XII, Exercise 2.10]{RY99});
\item
the time $T = 2 \gamma_{\tau^-}$ is contained in the first negative excursion to include an independent Poisson mark at rate $\frac{1}{2}$ per unit original time, which corresponds to some rate $\xi$ of marked excursions per unit local time, with $\xi$ to be determined.
\end{itemize}
By It\^o's description of Brownian excursion process and basic properties of Poisson processes, 
there is the well established argument \cite[Section 49]{RWvol2} that $L(0,T)$ is exponential distributed, and $L(0,T)$ is independent of all positive excursions and their marks.
It then follows from the decomposition \eqref{eq:localtdecomp} that the distribution of the local time process $(L(v, T), \, v \ge 0)$ is the same 
for all such random times $T$ whose $L(0,T)$ has the same rate $\xi$ per unit local time.
According to Levy's formula for the Laplace transform of $\tau_\ell$ (see e.g. \cite[Chapter II, Proposition 3.7]{RY99}),
for $\lambda >0$,
\begin{equation*}
\mathbb{P}(\tau_\ell < \lambda^{-1} \gamma_1) = \mathbb{E}e^{-\lambda \tau_\ell} = e^{-\ell \sqrt{2 \lambda}}.
\end{equation*}
This formula shows that for the present choice of Poisson marking rate $\lambda = 1/2$, 
the rate per unit local time is $\xi = \sqrt{2 \frac{1}{2}} = 1$, hence the rate of negative marked excursions per unit local time is $\frac{1}{2} \times 1 = \frac{1}{2}$.
As noted above, this is also the rate of excursions that reach $-1$ per unit local time. 
Consequently. the local time process on the right side of \eqref{eq:54} has the same distribution as the local time process described by the Ray-Knight identity \eqref{eq:RK0} for $x = 1$, hence the conclusion.

\quad (2) Suppose first that $T$ is either a fixed time, or a random time independent of $B$. 
Then conditionally given $T$ and the path of $B$ on $[0,T]$, the points $2 \gamma_n$ with $2 \gamma_n \le T$ are the points of a Poisson process with intensity $\frac{1}{2} 1(t \le T) dt$ on $(0, \infty)$.
By definition of the local time process as occupation density, the image of the measure $\frac{1}{2} 1(t \le T) dt$ via the continuous mapping $t \to B_t$ is $\frac{1}{2} L(x,T) dx$ on $(-\infty, \infty)$.
It follows from the mapping theorem for Poisson processes \cite[p. 18]{Kingman93} that conditionally given $T$ and the path of $B$ on $[0,T]$, 
the point process with points at $S_n: = B_{2 \gamma_n}$ for $n$ with $2 \gamma_n \le T$ is a Poisson process with intensity $\frac{1}{2}L(x,T)dx$ on $(-\infty, \infty)$.
That is to say, the point process 
\begin{equation*}
N_T(\cdot):= \sum_{n \ge 1} 1(2 \gamma_n \le T, \, B_{2 \gamma_n} \in \cdot),
\end{equation*}
is a Cox process with random intensity measure which is the occupation measure of $B$ on $[0,T]$ with continuous density $\frac{1}{2} L(x,T) dx$, $x \in (-\infty, \infty)$.

\quad This assertion is false for random times $T$ such that $T = 2 \gamma_{\tau^-}$, the time of the first Poisson sampling point $2 \gamma_n$ with $B_{2 \gamma_n} \le 0$: the restriction of $N_T(\cdot)$ to $(-\infty, 0]$ is not a Cox process since it has exactly one point. 
Still, by the independence of the processes of positive and negative excursions and their associated Poisson marks, it can be argued as above that for this $T = 2 \gamma_{\tau^-}$, the only way in which the sampling rule affects the distribution of marks over times $t < T$ when $B_t > 0$ is through the occupation measure of $B$ on $(0, \infty)$ up to time $T$, as encoded by the local time process $(L(v,T), \, v \ge 0)$.
So for $T = 2 \gamma_{\tau^-}$, the restriction of $N_T(\cdot)$ to $(0, \infty)$ is a Cox process as claimed. 
The identity in law \eqref{eq:55} then follows from \eqref{eq:54}.

\quad (3) It follows from part (2) and the generating function \eqref{eq:26} for a Cox process that
\begin{equation}
\label{eq:58}
\mathbb{E}z^{N_{\tiny \mbox{des}}(v)} = \mathbb{E} \left[-\frac{1}{2}(1-z)\int_0^v Q_0(2 \gamma_1, s)ds \right].
\end{equation}
Specializing the joint Laplace formula \eqref{eq:jointLaplace} to $\delta = 0$ and $\alpha = 0$, we get
\begin{equation}
\label{eq:59}
\begin{aligned}
\mathbb{E} \left[-\frac{\beta^2}{2}\int_0^v Q_0(2 \gamma_1, s)ds \right] &= \int_0^{\infty} \exp\left(-\frac{x}{2} \beta \tanh \beta v \right) \cdot \frac{1}{2} e^{-x/2}dx \\
& = (1 + \beta \tanh \beta v)^{-1},
\end{aligned}
\end{equation}
which by injecting into \eqref{eq:58} with $\beta^2 = 1-z$ yields the formula \eqref{eq:Laplace56}.
\end{proof}

\quad Now we provide several applications of \eqref{eq:Laplace56} and derive explicit formulas for the point process $\{M_1, \ldots, N_{\nu}\}$.
First by noting that $\nu = N_{\tiny \mbox{des}}(\infty)$, we get
\begin{align*}
\mathbb{E}z^{\nu} = \lim_{v \to \infty} \mathbb{E}z^{N_{\tiny \mbox{des}}(v)} = \lim_{v \to \infty} (1+ \sqrt{1-z} \tanh v \sqrt{1-z})^{-1} = (1+ \sqrt{1-z})^{-1},
\end{align*}
which recovers the formula \eqref{eq:nuPGF} for the probability generating function of $\nu$.
Further by expanding  \eqref{eq:Laplace56} in powers of $z$, we obtain the probabilities for $v \ge 0$,
\begin{align*}
& \mathbb{P}(N_{\tiny \mbox{des}}(v) = 0) = \frac{1}{2}(1+e^{-2v}), \quad \mathbb{P}(N_{\tiny \mbox{des}}(v) = 1) = \frac{1}{8}(1+ 4ve^{-2v} -e^{-4v}), \\
& \mathbb{P}(N_{\tiny \mbox{des}}(v) = 2) = \frac{1}{32}\left(2+ (-1+4v + 8v^2) e^{-2v} - (2+8v)e^{-4v}+e^{-6v}\right),
\end{align*}
and so on.
It develops that $\mathbb{P}( N_{\tiny \mbox{des}}(v)= k)$ is for each $k$ a signed linear combination of exponentials $e^{-2jv}$ for $j = 0, \ldots, k+1$ with coefficients that are polynomials in $v$.
But the exact form of the coefficients involved does not seem to be obvious.
These probabilities determine the law of $M_k$'s by the basic duality relation
\begin{equation*}
\mathbb{P}(M_k > v) = \mathbb{P}(N_{\tiny \mbox{des}}(v) < k) = \sum_{j = 0}^{k-1} \mathbb{P}(N_{\tiny \mbox{des}}(v) = j),
\end{equation*}
where by convention $M_k = \infty$ if $\nu < k$, so the event $\{M_k > v\}$ includes the event $\{M_k = \infty\} = \{\nu < k\}$.
Thus,
\begin{align*}
& \mathbb{P}(M_1 > v) = \frac{1}{2}(1+e^{-2v}), \quad \mathbb{P}(M_2 > v) = \frac{1}{8}(5+ 4(1+v)e^{-2v} -4e^{-4v}), \\
& \mathbb{P}(M_3 > v) = \frac{1}{32}\left(22+ (15+20v + 8v^2) e^{-2v} - (6+8v)e^{-4v}+e^{-6v}\right),
\end{align*}
and so on, where the constant coefficients obtained in the limit as $v \to \infty$ are the cumulative probabilities in the distribution of $\nu$,
that is,
\begin{align*}
&\mathbb{P}(M_1 = \infty) = \mathbb{P}(\nu = 0) = \frac{1}{2}, \quad
\mathbb{P}(M_2 = \infty) = \mathbb{P}(\nu \le 1) = \frac{1}{2} + \frac{1}{8} = \frac{5}{8}, \\
&\mathbb{P}(M_3 = \infty) = \mathbb{P}(\nu \le 2) = \frac{1}{2} + \frac{1}{8} + \frac{1}{16} = \frac{22}{32},
\end{align*}
and so on. 

\quad Note that the formula \eqref{eq:59} is the instance $x = 1$, $u = 0$ of the following more general consequence of 
the Ray-Knight identity \eqref{eq:RK0}: for $-x < 0 \le u \le v$,
\begin{equation}
\begin{aligned}
\mathbb{E} \left[-\frac{\beta^2}{2}\int_u^v Q_0(2 x\gamma_1, s)ds \right] & = 
\mathbb{E} \left[-\frac{\beta^2}{2}\int_0^{T_{-x}} 1(u < B_t \le v)dt \right] \\
& = \frac{1 + u\beta \tanh(v-u)\beta}{1+(u+x)\beta \tanh(v-u)\beta},
\end{aligned}
\end{equation}
where the second equality is read from \cite[p. 203 (2.7.1)]{BS96}.
When evaluated at $x = 1$ and $\beta = \sqrt{1-z}$, this identity gives the probability generating function of 
$N_{\tiny \mbox{des}}(u,v]:= \sum_{k=1}^\nu 1(u < S_k \le v) = \sum_{k=1}^\nu 1(u < M_k \le v)$,
the number of steps of the symmetric Laplace walk in $(u,v]$ before its first step into the negative half line:
\begin{equation}
\label{eq:uvPGF}
\mathbb{E}z^{N_{\tiny \mbox{des}}(u,v]} = \frac{1+u\sqrt{1-z} \tanh(v-u)\sqrt{1-z}}{1+ (u+1)\sqrt{1-z} \tanh(v-u)\sqrt{1-z}}.
\end{equation}
Expanding \eqref{eq:uvPGF} in powers of $z$ gives 
\begin{align*}
&\mathbb{P}(N_{\tiny \mbox{des}}(u,v] = 0) = \frac{1+u \tanh(v-u)}{1+ (1+u) \tanh(v-u)}, \\
&\mathbb{P}(N_{\tiny \mbox{des}}(u,v] = 1) = \frac{\tanh(v-u) + (v-u)(1-\tanh^2(v-u))}{2(1+ (1+u) \tanh(v-u))^2},
\end{align*}
and so on, with progressively more complex ratios of polynomials in $u, v$ and $\tanh(v-u)$.
Similarly, by expanding in powers of $1-z$, we get
\begin{equation*}
\mathbb{E}N_{\tiny \mbox{des}}(u,v] = v-u, \quad \mathbb{E}N^2_{\tiny \mbox{des}}(u,v] =\frac{2}{3}(v-u)^2(3+2u+v)+(v-u),
\end{equation*}
and so on. 
In the limiting case $v = \infty$, we derive the probability generating function of $N_{\tiny \mbox{des}}(u,\infty)$:
\begin{equation}
\mathbb{E}Z^{N_{\tiny \mbox{des}}(u,\infty)} = \frac{1 + u \sqrt{1-z}}{1 + (u+1)\sqrt{1-z}},
\end{equation}
which by expanding in powers of $z$ gives
\begin{align*}
&\mathbb{P}(N_{\tiny \mbox{des}}(u,\infty) = 0) = \frac{1+u}{2+u}, \quad 
\mathbb{P}(N_{\tiny \mbox{des}}(u,\infty) = 1) = \frac{1}{2(2+u)^2}, \\
&\mathbb{P}(N_{\tiny \mbox{des}}(u,\infty) = 2) = \frac{4+3u}{8(2+u)^3}, \quad 
\mathbb{P}(N_{\tiny \mbox{des}}(u,\infty) = 3) = \frac{10+14u+5u^2}{16(2+u)^4}, 
\end{align*}
and so on. 
In particular, we recover the formula \eqref{eq:maxMnu} for the law of $M_{\nu}$, the largest point of $\{S_1, \ldots, S_{\nu}\}$ by observing that
$\mathbb{P}(M_{\nu} > t) = 1 - \mathbb{P}(N_{\tiny \mbox{des}}(t,\infty) = 0)$.

\section{Renewal cluster process, squared Bessel process and Laplace walk}
\label{sc6}

\quad In this section, we study the counting process $N_W(v): = \sum_{k \ge 1} 1(W_k \le v)$, $v \ge 0$
for $(W_k, \, k \ge 1)$ defined by \eqref{eq:Wkrep} as the sequence of limits in law of $(M_{k,n} - M_{k,0}, \, k \ge 1)$
for the symmetric Laplace walk $S$.
We also make a connection between $(N_W(v), \, v \ge 0)$ and $(N_{\tiny \mbox{des}}(v), \, v \ge 0)$, the counting process of the Laplace walk before the first descending ladder time $\tau^-: = \nu+1$, 
and then prove Theorem \ref{thm:main} as a byproduct.

\quad For a general random walk, the counting process $(N_W(v), \, v \ge 0)$ is the sum of two independent components
\begin{equation}
N_W(v) = N^{\uparrow}(v) + N^{\downarrow}(v) \quad \mbox{for } v \ge 0,
\end{equation}
where $N^{\uparrow}(v):= \sum_{i =1}^\infty 1(S^{\uparrow}_i \le v)$ and $N^{\downarrow}(v):= \sum_{i =1}^\infty 1(-S^{\downarrow}_i \le v)$ are counting processes of the upward and downward Feller chains (see Section \ref{sc21}).
Moreover, if increment distribution of the walk is symmetric and continuous, these two counting processes are identically distributed. 
By Tanaka's decomposition of $S^{\uparrow}$ and $S^{\downarrow}$ (Lemma \ref{lem:clusterTanaka}), 
each of the counting processes $N^{\uparrow}$ and $N^{\downarrow}$ is a renewal cluster process. 
In general, the laws of these renewal cluster processes may be complicated.
For the symmetric Laplace walk, the structure of these counting processes is greatly simplified because the underlying renewal processes are Poisson. 

\quad Recall that $(Q_{\delta}(X, t), \, t \ge 0 )$ is a squared Bessel process of dimension $\delta$ with initial state $X$,
and $\gamma_1 < \gamma_2 < \cdots$ are points in a Poisson process with rate $1$ on $(0, \infty)$.
The following theorem identifies the counting process $(N_W(v), \, v \ge 0)$ as a Poisson cluster process, which in turn is
a Cox process driven by some functional of the $\mbox{BESQ}_4$ process. 

\begin{theorem}
\label{thm:Poissoncluster}
For $\lambda > 0$, let $(N_{\lambda}(v), \, v \ge 0)$ be the Poisson cluster process, with cluster centers at the points 
$\lambda^{-1} \gamma_1 < \lambda^{-1} \gamma_2 < \cdots$ of a Poisson process with rate $\lambda$ on $(0,\infty)$,
and a cluster on $[\lambda^{-1}\gamma_k, \infty)$ for each $k$ that is a copy of $N_{\tiny \mbox{des}}$.
That is,
\begin{equation}
N_{\lambda}(v):= \sum_{k = 1}^\infty \sum_{i = 0}^{\nu_k} 1(\lambda^{-1} \gamma_k + M_{k,i} \le v) \quad \mbox{for } v \ge 0,
\end{equation}
where $(M_{k,i}, \, 0 \le i \le \nu_k)$ for $k = 1, 2, \ldots$ is a sequence of independent copies of $(M_i, \, 0 \le i \le \nu)$,
the process of order statistics of $(S_i, \, 0 \le i \le \nu)$ for the symmetric Laplace walk $S$. 
Then
\begin{enumerate}[itemsep = 3 pt]
\item
The Poisson cluster process $(N_{\lambda}(v), \, v \ge 0)$ is a Cox process driven by
\begin{equation}
\label{eq:Qdiffstart}
\left(\frac{1}{2} Q_{2 \lambda}(2\gamma_{\lambda},v), \, v \ge 0 \right) \stackrel{d}{=} \left(\frac{1}{2} Q_{2 \lambda}(0,1+v), \, v \ge 0 \right),
\end{equation}
where $\gamma_{\lambda}$ is understood as a gamma random variable with density $\frac{1}{\Gamma(\lambda)} x^{\lambda-1}e^{-x}$, $x > 0$.
Consequently, for each $v>0$ the probability generating function of $N_{\lambda}(v)$ is
\begin{equation}
\label{eq:NlamPGF}
\mathbb{E}z^{N_{\lambda}(v)} = (\cosh v\sqrt{1-z} + \sqrt{1-z} \sinh v\sqrt{1-z})^{-\lambda}.
\end{equation}
\item
There are the identities in law of processes
\begin{align}
& (N^{\uparrow}(v), \, v \ge 0) \stackrel{d}{=} (N^{\downarrow}(v), \, v \ge 0) \stackrel{d}{=} (N_1(v), \, v \ge 0), \\
& (N_W(v), \, v \ge 0) \stackrel{d}{=} (N_2(v), \, v \ge 0).
\end{align}
So the probability generating functions of $N^{\uparrow}(v)$ and $N_W(v)$ are given by \eqref{eq:NlamPGF} for $\lambda = 1$ and $\lambda = 2$ respectively.
\item
The points $W_1 < W_2 < \ldots$ of the Cox process $N_W$ driven by $\left(\frac{1}{2} Q_{4}(0,1+v), \, v \ge 0 \right)$ may be constructed as the differences $T_2 - T_1 < T_3 - T_2 < \cdots$ for $0< T_1 < T_2 < \cdots$ the points of a Cox process driven by $\left(\frac{1}{2} Q_{4}(0,v), \, v \ge 0 \right)$.
\end{enumerate}
\end{theorem}

\begin{proof}
(1) Attaching i.i.d.  clusters to the points of a Poisson process is a well known mechanism for generating an infinitely divisible point process, whose L\'evy-Khintchine representation is easily expressed in terms of the probability generating function of counts of points in the clusters (see e.g. \cite[Section 3.3]{Kal17}).
For the Poisson cluster process $N_{\lambda}$, the probability generating function of clusters is given by \eqref{eq:Laplace56},
so
\begin{align*}
\mathbb{E}z^{N_{\lambda}(v)} &= \exp\left( - \lambda \int_0^v \left(1- \mathbb{E}z^{1+N_{\tiny \mbox{des}}(u)}\right)du\right) \\
& = \exp\left(-\lambda \int_0^v 1 - z(1+ \sqrt{1-z}\tanh u \sqrt{1-z})^{-1} du \right) \\
& = (\cosh v\sqrt{1-z} + \sqrt{1-z} \sinh v\sqrt{1-z})^{-\lambda},
\end{align*}
where in the first equation $M_0 = S_0$ is counted, hence $\sum_{0 \le i \le \nu} 1(S_i \le v) = 1+N_{\tiny \mbox{des}}(u)$,
and the last equation is obtained by elementary calculus
\begin{equation*}
\frac{d}{dv} \log(\cosh v\beta + \beta \sinh v\beta)|_{\beta = \sqrt{1-z}} = 1 - \frac{1-\beta^2}{1 + \beta \tanh v\beta}|_{\beta = \sqrt{1-z}}.
\end{equation*}
Further by the Laplace transform \eqref{eq:jointLaplace}, we get the generating function for a Cox process driven by 
$\left(\frac{1}{2} Q_{2 \lambda}(2\gamma_{\lambda},v), \, v \ge 0 \right)$:
\begin{align*}
& \quad \,\, \mathbb{E}\exp \left[-\frac{1}{2} \beta^2\int_0^v Q_{2 \lambda}(2\gamma_{\lambda},u)du \right]_{\beta = \sqrt{1-z}} \\
& = (\cosh v \beta)^{-\lambda} \int_0^{\infty}\frac{1}{2 \Gamma(\lambda)} \left(\frac{x}{2}\right)^{\lambda-1} \exp\left(-\frac{x}{2}(1 + \beta \tanh v \beta) \right) dx|_{\beta = \sqrt{1-z}} \\
& = (\cosh v \beta)^{-\lambda} (1+\beta \tanh v \beta)^{-\lambda}|_{\beta = \sqrt{1-z}},
\end{align*}
which identifies the Poisson cluster process $N_{\lambda}$ as a Cox process driven by $\left(\frac{1}{2} Q_{2 \lambda}(2\gamma_{\lambda},v), \, v \ge 0 \right)$.
The identity in law \eqref{eq:Qdiffstart} follows from the Markov property of $\mbox{BESQ}_{2 \lambda}$, 
and the fact that $2 \gamma_{\lambda} \stackrel{d}{=} Q_{2 \lambda}(0,1)$
(see e.g. \cite[Chapter XI, Corollary 1.4]{RY99}).

\quad (2) The fact that $N^{\uparrow}$ is a renewal cluster process follows from Lemma \ref{lem:clusterTanaka}, Tanaka's decomposition of $S^{\uparrow}$.
By Lemma \ref{lem:Laplacedes}, the cluster centers of $N^{\uparrow}$ are the points of a Poisson process with rate $1$.
As a result, $N^{\uparrow}$ is a Poisson cluster process with rate $1$.
As a sum of two independent copies of $N^{\uparrow}$, the process $N_W$ is a Poisson cluster process with rate $2$. 

\quad (3) In view of the Markov property of $\mbox{BESQ}_4$, to establish this representation it suffices to show that the time of the first point $T_1$ in a Cox process driven by $\left(\frac{1}{2} Q_{4}(0,v), \, v \ge 0 \right)$ is such that
\begin{equation}
\label{eq:67}
Q_4(0,T_1) \stackrel{d}{=} 2 \gamma_2 \stackrel{d}{=} Q_4(0,1).
\end{equation}
The second equality follows from the general fact that $2 \gamma_{\lambda} \stackrel{d}{=} Q_{2 \lambda}(0,1)$, $\lambda > 0$ as proved in part (1).
By \eqref{eq:214}, the Laplace transform of $Q_4(0,T_1)$ is 
\begin{align*}
\mathbb{E}e^{-\alpha Q_4(0,T_1)} &= \frac{1}{2} \int_0^{\infty} \mathbb{E}\left[ e^{-\alpha Q_4(0,t)} Q_4(0,t) e^{-\frac{1}{2} \int_0^t Q_4(0,v)dv} \right] dt \\
& = \frac{1}{2} \int_0^{\infty} \int_{x > 0} xe^{-\alpha x} \mathbb{E}\left(e^{-\frac{1}{2} \int_0^t Q_4(0,v)dv} \bigg| Q_4(0,t) = x\right) \mathbb{P}(Q_4(0,t) \in dx)dt \\
& = \frac{1}{2} \int_0^{\infty} \int_{0}^\infty xe^{-\alpha x} \left(\frac{t}{\sinh t}\right)^2 \exp\left(\frac{x}{2t}(1 - t \coth t) \right)\cdot \frac{x}{4t^2}\exp\left(-\frac{x}{2t} \right) dx dt \\
& = 2 \int_0^\infty \frac{1}{\sinh^2t} \frac{1}{(2 \alpha + \coth t)^3}dt = \frac{1}{(1+2\alpha)^2},
\end{align*}
where in the third equality the expression for $\mathbb{E}(e^{-\frac{1}{2} \int_0^t Q_4(0,v)dv} \,|\, Q_4(0,t) = x)$ is read from \cite[(2.m)]{PY82}.
This proves the first equality in \eqref{eq:67}.
\end{proof}

\quad Note that the key to the proof of Theorem \ref{thm:Poissoncluster} (3) is the fact that $Q_4(0, T_1) \stackrel{d}{=} 2 \gamma_2$.
This identity in law can also be recognized by the Poisson embedding of the symmetric Laplace walk in Brownian motion, 
the Williams decomposition of Brownian motion at its minimum, and the Ray-Knight-Williams description of local times near a local minimum in terms of the $\mbox{BESQ}_4$ process.
We will further discuss this aspect in Section \ref{sc7}.

\quad The next corollary identifies the law of the point process $\{W_1, W_2, \ldots\}$ with that corresponding to the order statistics of a `long' or a `high' excursion.
\begin{corollary}
\label{coro:identify2}
Let $N_W(v): = \sum_{k=1}^{\infty} 1(W_k \le v)$, $v \ge 0$ be the counting process of $(W_k, \, k \ge 1)$ defined by \eqref{eq:Wkrep} as the sequence of limits in law of $(M_{k,n} - M_{k,0}, \, k \ge 1)$
for the symmetric Laplace walk $S$.
Let $N_{\tiny \mbox{des}}(v):= \sum_{k=1}^{\infty} 1(S_k \le v, \, k< \tau^-)$ be the counting process of the order statistics $\{M_1, \ldots, M_{\tau^- - 1}\}$ of $S$ prior to $\tau^-$.
Then there is the convergence in total variation norm
\begin{equation}
(N_{\tiny \mbox{des}}(v), \, v \ge 0 \,|\, A_m) \stackrel{d}{\longrightarrow} (N_W(v), \, v \ge 0) \quad \mbox{as } m \to \infty,
\end{equation}
where $(A_m, \, m \ge 1)$ is either of the following two sequences of events: $A_m = \{\tau^- > m\}$ (`long' excursion)
or $A_m = \{\max_{k < \tau^-}S_k > m\}$ (`high' excursion).
For each finite $K$, there is also the convergence in total variation norm
\begin{equation}
(M_k,, 1 \le k \le K \,|\, A_m) \stackrel{d}{\longrightarrow} (W_k, \, 1 \le k \le K) \quad \mbox{as } m \to \infty.
\end{equation}
\end{corollary}
\begin{proof}
By construction, $W_k$ is the $k^{th}$ order statistics of $\{-S^{\downarrow}_n, \, n \ge 0\} \cup \{S^{\uparrow}_n, \, n \ge 1\}$,
where $(S^{\downarrow}_n, \, n \ge 0)$ and  $(S^{\uparrow}_n, \, n \ge 1)$ are two Feller chains generated from the walk $S$.
Since the increments of the walk $S$ are i.i.d. with continuous density, 
$(-S^{\downarrow}_n, \, n \ge 0)$ and  $(S^{\uparrow}_n, \, n \ge 1)$ are independent and identically distributed as the Doob-$h$ transform of the walk $S$ for some harmonic function $h$.
It is well known \cite{GS79} that for a long walk excursion or a high walk excursion, 
the contributions to the counting process $(N_{\tiny \mbox{des}}(v), \, v \ge 0)$ come from two ends of the excursion.
Since the increments of the walk $S$ are i.i.d. with continuous density, the beginning and the end of this excursion are asymptotically independent and identically distributed as the Doob-$h$ transform of the walk $S$ with the harmonic function $h$, see \cite{BD94}.
In particular, for each $v \ge 0$,
$\mathbb{E}(z^{N_{\tiny \mbox{des}}(v)}\,|\, A_m) \longrightarrow (\mathbb{E}z^{N^{\uparrow}(v)})^2$
as $m \to \infty$,
which is equal to $\mathbb{E}z^{N_W(v)}$ by Theorem \ref{thm:Poissoncluster} (1)(2).
\end{proof}

\quad To conclude this section, we prove Theorem \ref{thm:main}.
\begin{proof}[Proof of Theorem \ref{thm:main}]
(1) According to Corollary \ref{coro:identify2} with $A_m = \{\tau^- > m\}$, 
we identify the distribution of $(N_W(v), \, v \ge 0)$ with the limiting distribution of $(N_{\tiny \mbox{des}}(v), \, v \ge 0 \,|\, \nu \ge m)$ as $m \to \infty$.
Now by Lemma \ref{lem:idDelta}, we get
\begin{equation*}
W_k \stackrel{d}{=} \sum_{j=1}^k \frac{\varepsilon_j}{2 Y^{\uparrow}_j} \quad \mbox{for } k \ge 1,
\end{equation*}
where $(Y^{\uparrow}_j, \, j \ge 1)$ with $Y^{\uparrow}_1 = 1$ is the Markov chain with transition probabilities \eqref{eq:39},
and $(\varepsilon_j, \, j \ge 1)$ is a sequence of i.i.d. exponential variables independent of $Y^{\uparrow}$.
We conclude by noting that $(2 Y_j, \, j \ge 1) \stackrel{d}{=} (S^{\pm, 0}_{2j}, j \ge 1)$ where $S^{\pm,0}$ is a simple symmetric random walk on the nonnegative integers with absorption at $0$, and thus $(2 Y^{\uparrow}_j, \, j \ge 1) \stackrel{d}{=} (S^{\pm \uparrow}_{2j}, j \ge 1)$ as the Doob-$h$ transforms.
The tail distribution \eqref{eq:tailDj} follows from the representation $D_j := W_{j} - W_{j-1} = \varepsilon_j/S^{\pm \uparrow}_{2j}$.

\quad (2) The statements for $N_W$ follows directly from Theorem \ref{thm:Poissoncluster}.
By the formula \eqref{eq:215}, the tail generating function for $D_j$'s is 
\begin{equation*}
\sum_{j = 1}^\infty \mathbb{P}(D_j > v)z^{j-1} = \frac{1}{1-z} \mathbb{E}\left[-\frac{1}{2} \int_0^v Q_4(X, u) du \right],
\end{equation*}
where $X$ is the time of the first point in a Cox process driven by $\left(\frac{1-z}{2}Q_4(0,v), \, v \ge 0\right)$.
A similar argument as for the first identity in \eqref{eq:67} shows that $X \stackrel{d}{=} \frac{2}{\sqrt{1-z}} \gamma_2$.
Consequently,
\begin{align*}
\sum_{j = 1}^\infty \mathbb{P}(D_j > v)z^{j-1} &= \frac{1}{1-z} \mathbb{E}\left[-\frac{1}{2} \int_0^v Q_4\left(\frac{2}{\sqrt{1-z}}\gamma_2, u\right) du \right] \\
& =  \frac{1}{1-z} \int_{x > 0} \mathbb{E}\left[-\frac{1}{2} \int_0^v Q_4(x,u) du \right] \mathbb{P}\left( \frac{2}{\sqrt{1-z}}\gamma_2 \in dx\right)\\
& = \frac{1}{1-z} \int_0^{\infty} (\cosh v)^{-2} \exp \left(-\frac{x}{2} \tanh v\right) \cdot \frac{1-z}{4} x \exp\left(-\frac{x}{2}\sqrt{1-z} \right) dx \\
& = (\cosh v)^{-2} (\sqrt{1-z} + \tanh v)^{-2},
\end{align*}
where the third equality follows from the Laplace transform \eqref{eq:jointLaplace}.
This yields \eqref{eq:tailPGFD}.
\end{proof}

\section{Brownian embedding of the Laplace walk and path decomposition}
\label{sc7}

\quad In this final section, we study the path decomposition of the symmetric Laplace walk embedded in Brownian motion, which sheds light on previous constructions of $(W_k, \, k \ge 1)$ as the sequence of limits in law of $(M_{k,n} - M_{k,0}, \, k \ge 1)$ for the walk $S$.
It is well known \cite{Sk65} that every random walk with mean zero and finite variance $\sigma^2$ per step may be 
embedded in a Brownian motion $(B_t, \, t \ge 0 )$ as $S_k:= B_{T_k}$, 
where $0 = T_0 \le T_1 \le \cdots$ is an increasing sequence of stopping times of $B$, with the $(T_k- T_{k-1}, \, k \ge 1)$ i.i.d. as $T_1$, and $\mathbb{E}T_1 = \sigma^2$. 
Three important examples are
\begin{itemize}[itemsep = 3 pt]
\item[$(1)$]
simple symmetric walk, with $T_k = \inf\{t > T_{k-1}: |B_t - B_{T_{k-1}}| = 1\}$; 
\item[$(2)$]
symmetric Gaussian random walk, with $T_k = k$;
\item[$(3)$]
symmetric Laplacian random walk, with $T_k = 2 \gamma_k$, where $\gamma_k = \sum_{i = 1}^{k} \varepsilon_i$ for 
$(\varepsilon_k, \, k \ge 1)$ a sequence of i.i.d. standard exponential variables.
\end{itemize}
In our previous work \cite{PT20}, the Brownian embeddings of (1) the simple symmetric walk, and (2) the Gaussian walk 
have been thoroughly studied. 

\quad In the sequel, it is assumed that the probability space on which the random walk $(S_n, \, n \ge 0)$ is defined is rich enough to allow for an embedding in a Brownian motion $(B_t, \, t \ge 0)$ defined on the same probability space.
There is one more variable to be entered into the mix:
\begin{equation}
M_{-,n}:= \min_{0 \le t \le T_n} B_t.
\end{equation}
So by definition, $M_{-,n} \le M_{0,n} \le \cdots \le M_{n,n}$ are ranked values of Brownian motion on $[0,T_n]$ evaluated at a grid of $n+2$ random times, the times $0= T_0 \le T_1 \le \cdots \le T_n$ at which the random walk $S_k:=B_{T_k}$ is
embedded, and one extra random time $\tau[0,T_n]$, the almost surely unique random time at which $B$ attains its minimum on $[0, T_n]$.
It is a key observation that the differences of the order statistics $M_{k,n} - M_{k,0}$ can be expressed in terms of 
the ranked heights $M_{k,n} - M_{-,n}$ of the random walk sample points above the minimum of $B$ on $[0,T_n]$:
\begin{equation}
M_{k,n} - M_{0,n} = (M_{k,n} - M_{-,n}) - (M_{0,n} - M_{-,n}).
\end{equation}

\quad The distribution of $M_{0,n} - M_{-,n}$ was studied in \cite{AGP95} for a Gaussian random walk, 
where $M_{0,n} - M_{-,n}$ was interpreted as the discretization error in the Euler scheme to approximate a reflected Brownian motion.
The main idea was that when the Gaussian walk $S$ is embedded in Brownian motion with $S_k = B_{T_k}$ for $T_k = k$,
the discretization error $M_{0,n} - M_{-,n}$ as $n \to \infty$ is determined by the behavior of Brownian motion around its minimum time, which may be described as two independent copies of a $\mbox{BES}_3(0)$ process, joined back to back.
Extending this argument, it was shown in \cite{PT20} that for the symmetric Gaussian walk, 
there is the convergence in joint distributions for each finite $K$,
\begin{equation*}
(M_{k,n} - M_{-,n}, \, 0 \le k \le K) \stackrel{d}{\longrightarrow} (M_{k,\infty}, \, 0 \le k \le K),
\end{equation*}
where $(M_{k,\infty}, \, k \ge 0)$ is the point process of the order statistics of $(\check{R}_{U+z}, \, z \in \mathbb{Z})$, 
with $R'$ a two-sided $\mbox{BES}_3(0)$ process, and $U$ uniform on $[0,1]$ independent of $\check{R}$.
However, it is a difficult problem to describe even the law of $M_{0, \infty}$ at all explicitly.

\quad It has been known for a long time that computations in the fluctuation theory of random walks are often surprisingly difficult in the Gaussian case, while surprisingly easy in the exponential case. 
The good feature of exponentially distributed increments is that the memoryless property of the exponential distribution implies that all overshoots of levels are exponentially distributed. 
The simple structure of the path decomposition at the minimum of Brownian motion up to an independent exponential time, first emphasized by Williams \cite{Williams70, Williams}, has enabled diverse developments of the rich probabilistic structure of Brownian motion sampled at random times generated by an independent homogeneous Poisson process,
see e.g. \cite[Sections 7.7--7.8]{Pitman06}.
In particular, as observed in \cite[p. 892]{AGP95}, this structure provides a simple algorithm for efficient sampling the values of a reflected Brownian motion at the times generated by an independent homogeneous Poisson process.
The following analog of \cite[Theorem 5.1]{PT20} for Poisson sampling is also easily established by the method of \cite{AGP95}.

\begin{theorem}
\label{thm:pathdecomp}
Let $(S_k, \, k \ge 0)$ be a random walk with i.i.d. symmetric Laplace increments
embedded in Brownian motion $(B_t, \, t \ge 0)$ as $S_k:=B_{T_k}$ with $T_k = 2 \gamma_k$ for $\gamma_k = \sum_{i = 1}^k \varepsilon_i$ for $\varepsilon_1, \varepsilon_2, \ldots$ a sequence of i.i.d. standard exponential variables independent of $B$.
Let $(M_{k,n}, \, 0 \le k \le n)$ be the sequence of order statistics of the $n$-step walk $(S_k, \, 0 \le k \le n)$,
and $M_{-,n}:=\min_{0 \le t \le T_n} B_t$.
Then for each finite $K$, there is the convergence of joint distributions
\begin{equation}
(M_{k,n} - M_{-,n}, \, 0 \le k \le K) \stackrel{d}{\longrightarrow} (M_{k,\infty}, \, 0 \le k \le K),
\end{equation}
where $(M_{k, \infty}, \, k \ge 0)$ is the sequence of order statistics of values of $(R_3(2\gamma_k), \, k = 1,2,\ldots)$
and $(\widehat{R}_3(2 \widehat{\gamma}_k), \, k = 1,2,\ldots)$, 
with two independent $\mbox{BES}_3(0)$ process $R_3$ and $\widehat{R}_3$ assumed to be independent of 
$(\gamma_k, \, k \ge 1)$ and $(\widehat{\gamma}_k, \, k \ge 1)$, the points of two independent Poisson process with rate $1$.
\end{theorem}

\quad Since the $\mbox{BES}_3(0)$ process $R$ is transient, its ultimate local time process $(L_R(v, \infty), \, v \ge 0)$ is almost surely finite for each $v \ge 0$.
Theorem \ref{thm:pathdecomp} then applies to show that the ultimate point process of levels above the minimum derived from a symmetric Laplace walk is a Cox process, whose intensity is described by the Ray-Knight-Williams description of Brownian local times, as indicated in the following corollary.
This result also gives a simplified proof of Theorem \ref{thm:Poissoncluster} (3), or equivalently the first identity in law in \eqref{eq:67}.

\begin{corollary}
\label{coro:last}
Under the setting of Theorem \ref{thm:pathdecomp}, the limiting point process of levels above the minimum with ordered points $(M_{k,\infty}, \, k \ge 0)$ is a Cox process on the positive half line driven by
\begin{equation}
\frac{1}{2} Q_4(0, v): = \frac{1}{2}\left(L_R(v, \infty) + L_{\widehat{R}}(v, \infty)\right) \quad v \ge 0,
\end{equation}
where $Q_4(0, \cdot)$ is a $\mbox{BESQ}_4$ process starting at $0$.
The law of this point process is infinitely divisible.
In particular, the decomposition of the limit levels above the minimum provides a decomposition of limit levels into 
the sum of two independent copies of the Cox process driven by 
\begin{equation}
\frac{1}{2}Q_2(0,v):= \frac{1}{2}L_R(v, \infty) \quad v \ge 0,
\end{equation}
where $Q_2(0, \cdot)$ is a $\mbox{BESQ}_2$ process starting at $0$.
\end{corollary}

\medskip
{\bf Acknowledgment}: We thank G. Schehr for stimulating discussions at the early stage of this work.
Tang gratefully acknowledges financial support through an NSF grant DMS-2113779 and through a start-up grant at Columbia University.

\bibliographystyle{abbrv}
\bibliography{unique}
\end{document}